\documentclass[11pt]{article}
\usepackage[usenames]{color} 
\usepackage[applemac]{inputenc}
\usepackage{amssymb} 
\usepackage{amsmath} 
\usepackage{mathabx}
\usepackage{amsthm}
\usepackage[all]{xy}
\usepackage{enumerate}
\usepackage{enumitem}
\usepackage[T1]{fontenc}
\usepackage[left=3.5cm,right=3.5cm,top=3cm,bottom=3.5cm]{geometry}
\usepackage{authblk}
\usepackage{cleveref}
\usepackage{filecontents}
\usepackage{graphics} 
\usepackage{graphicx}
\usepackage{pstricks,pst-node} 
\usepackage{tikz} 
\usepackage{xcolor}
\usepackage{shuffle}
\usepackage{lipsum}

\newcommand\blfootnote[1]{%
  \begingroup
  \renewcommand\thefootnote{}\footnote{#1}%
  \addtocounter{footnote}{-1}%
  \endgroup
}

\begin{document}
\theoremstyle{plain}
\newtheorem{thm}{Theorem}[section]
\newtheorem{lem}[thm]{Lemma}
\newtheorem{prop}[thm]{Proposition}
\newtheorem{cor}[thm]{Corollary}
\newtheorem{conj}[thm]{Conjecture}
\newtheorem{Question}[thm]{Question}

\theoremstyle{definition}
\newtheorem{deftn}[thm]{Definition}
\newtheorem{ex}[thm]{Example}

\theoremstyle{remark}
\newtheorem{rk}[thm]{Remark}

\newtheorem{Fact}{Evidence}
\newtheorem{assump}{Assumption}[section]
\renewcommand{\theassump}{\Alph{assump}}
\newtheorem{conjintro}{Conjecture}
\newtheorem{thmintro}{Theorem}

\numberwithin{equation}{section}

\newcommand{\A}{\mathcal{A}}
\newcommand{\B}{\mathcal{B}}
\newcommand{\C}{\mathcal{C}}
\newcommand{\CQ}{\mathcal{C}_Q}
\newcommand{\Cw}{\mathcal{C}_w}
\newcommand{\yjh}{\hat{y_j}}
\newcommand{\ylh}{\hat{y_1}}
\newcommand{\ynh}{\hat{y_n}}
\newcommand{\ykh}{\hat{y_k}}
\newcommand{\mjh}{\hat{\mu_j}}
\newcommand{\mkh}{\hat{\mu_k}}
\newcommand{\bmjh}{\hat{\boldsymbol{\mu}_j}}
\newcommand{\bmkh}{\hat{\boldsymbol{\mu}_k}}
\newcommand{\bmu}{\boldsymbol{\mu}}
\newcommand{\p}{\mathbb{P}}
\newcommand{\W}{{\bf W}}
\newcommand{\Aqnw}{\mathcal{A}_q (\mathfrak{n}(w))}
\newcommand{\Aqn}{\mathcal{A}_q (\mathfrak{n})}
\newcommand{\Yia}{Y_{i,a}}
\newcommand{\M}{{\bf M}}
\newcommand{\G}{{\bf G}}
\newcommand{\s}{\mathcal{S}}
\newcommand{\hs}{\mathcal{H}_{\mathcal{S}}}
\newcommand{\ps}{\mathcal{P}_{\mathcal{S}}}
\newcommand{\hw}{\mathcal{H}_{w}}
\newcommand{\pw}{\mathcal{P}_{w}}
\newcommand{\zn}{\mathbb{Z}^N}
\newcommand{\hyp}{\mathcal{H}}
\newcommand{\ds}{\Delta_{\mathcal{S}}}
\newcommand{\hgt}{\text{ht}}
\newcommand{\wt}{\text{wt}}
\newcommand{\Vect}{\text{Vect}}

\newcommand{\Pik}{P_{\text{in}(k)}}
\newcommand{\Pok}{P_{\text{out}(k)}}
\newcommand{\Pij}{P_{\text{in}(j)}}
\newcommand{\Poj}{P_{\text{out}(j)}}
\newcommand{\Pjt}{\tilde{P}_j}
\newcommand{\Plt}{\tilde{P}_l}
\newcommand{\Pkt}{\tilde{P}_k}
\newcommand{\bs}{\mathbf{s}}

\newcommand{\Address}{ \textsc{ \\ Universit\'e de Paris, Sorbonne Universit\'e, CNRS \\  Institut de Math\'ematiques de Jussieu-Paris Rive Gauche, UMR 7586, \\ F-75013 Paris \\ FRANCE} \\
 E-mail: \texttt{elie.casbi@imj-prg.fr}}

\title{\Large{ {\bf EQUIVARIANT MULTIPLICITIES OF SIMPLY-LACED TYPE FLAG MINORS}}}

 \author{ELIE CASBI}

     \date{}

\maketitle

 \begin{abstract}
Let $\mathfrak{g}$ be a finite simply-laced type simple Lie algebra. Baumann-Kamnitzer-Knutson recently defined an algebra morphism $\overline{D}$ on the coordinate ring $\mathbb{C}[N]$ related to Brion's equivariant multiplicities via the geometric Satake correspondence. This map is known to take distinguished values on the elements of the MV basis corresponding to smooth MV cycles, as well as on the elements of the dual canonical basis corresponding to Kleshchev-Ram's strongly homogeneous modules over quiver Hecke algebras. In this paper we show that when $\mathfrak{g}$ is of type $A_n$ or $D_4$, the map $\overline{D}$ takes similar distinguished values on the set of all flag minors of $\mathbb{C}[N]$, raising the question of the smoothness of the corresponding MV cycles.  We also exhibit certain relations between the values of $\overline{D}$ on flag minors belonging to the same standard seed, and we show that in any $ADE$ type these relations are preserved under cluster mutations from one standard seed to another. The proofs of these results partly rely on Kang-Kashiwara-Kim-Oh's monoidal categorification of the cluster structure of $\mathbb{C}[N]$ via representations of quiver Hecke algebras. 
 \end{abstract}

 \blfootnote{Keywords: flag minors, quiver Hecke algebras, equivariant multiplicities}

\blfootnote{MSC 16G10, 20G05}

\setcounter{tocdepth}{1}
\tableofcontents

 \section{Introduction}
 
 Let $\mathfrak{g}$ be a finite simply-laced type simple Lie algebra and let  $\mathfrak{n}$ denote the nilpotent subalgebra arising from a triangular decomposition of $\mathfrak{g}$. We consider the ring $\mathbb{C}[N]$ of regular functions on the algebraic group $N$ associated to $\mathfrak{n}$. The study of good bases of $\mathbb{C}[N]$, as well as its quantized version $\Aqn$, has been an intensively investigated topic since the works of  Kashiwara \cite{Kashicrystal} and Lusztig \cite{Lusztig} in the early 90's. Kashiwara \cite{Kashicrystal} introduced the notion of \textit{crystal} as a combinatorial model describing the structure of the irreducible finite-dimensional representations of the quantum group $U_q(\mathfrak{g})$ associated to $\mathfrak{g}$. He defined the lower global basis (resp. upper global basis) using the crystal structure on $U_q^{-}(\mathfrak{g})$ (resp. on the quantum coordinate ring $\Aqn$). Lusztig \cite{Lusztig} used certain categories of perverse sheaves on quiver varieties to define the canonical basis (resp. dual canonical basis) of $U_q^{-}(\mathfrak{g})$ (resp. $\Aqn$). Grojnowski-Lusztig \cite{GrojLus} and Kashiwara-Saito \cite{KashiSai} proved that the dual canonical basis and the upper global basis coincide. Several other remarkable bases of $\mathbb{C}[N]$ have been discovered since then, such as the dual semicanonical basis introduced by Lusztig \cite{LusztigAdv00} or the MV basis, constructed after the discovery of the geometric Satake correspondence by Mirkovi\'c-Vilonen  \cite{MV}. 
 
  Berenstein-Zelevinsky \cite{BZstring} observed that the dual canonical basis of $\mathbb{C}[N]$ had interesting multiplicative properties. This was one of the main motivations for the introduction of cluster algebras by Fomin-Zelevinsky \cite{FZ1}. These are defined as certain commutative subalgebras of the field of rational functions $\mathbb{Q}(x_1, \ldots , x_N)$ where $x_1 , \ldots , x_N$ are algebraically independent variables. They are generated by certain distinguished generators called \textit{cluster variables} that are grouped into overlapping finite sets of fixed cardinality $N$ called \textit{clusters}. The monomials involving variables of the same cluster are called  \textit{cluster monomials}. The cluster variables can be constructed from the variables $x_1, \ldots , x_N$ by performing an inductive procedure called \textit{mutation}. The initial data of this procedure consists in the $N$ independent variables $x_1, \ldots , x_N$ together with a quiver $Q$ with $N$ vertices and without any loop or $2$-cycle. Such a data is called a \textit{seed}.  For every $k \in \{1, \ldots , N\}$, one defines a new variable $x'_k$ entirely determined by the $x_j$ and $Q$, as well as a new quiver $Q'_k$. This yields a new seed, given by the variables $x_1, \ldots , x_{k-1}, x'_k, x_{k+1}, \ldots , x_N$ and the quiver $Q'_k$. One of the first key points of cluster theory is the involutivity of this procedure, i.e. mutating this new seed in the same direction $k$ transforms it back into the initial seed.  Thus one can iterate this by applying arbitrary sequences of mutations. Fomin-Zelevinsky \cite{FZ2} provided a Dynkin-type classification of the initial quivers $Q$ for which this process produces only a finite number of distinct seeds. 
  
   It was shown by Geiss-Leclerc-Schr\"oer \cite{GLS} that the coordinate ring $\mathbb{C}[N]$ associated to a simply-laced finite type Lie algebra $\mathfrak{g}$  carries the structure of a cluster algebra. Their work strongly relies on categorification techniques using representations of preprojective algebras. The mutations arise from the study of certain $T$-systems called \textit{determinantal identities} relating  unipotent minors. Geiss-Leclerc-Schr\"oer also explicitly construct a certain family of seeds called \textit{standard seeds} parametrized by the set of reduced expressions of the longest element $w_0$ of the Weyl group $W$ of $\mathfrak{g}$. The cluster variables of the standard seeds are certain special cases of unipotent minors called \textit{flag minors}. 
    
After the works of Geiss-Leclerc-Schr\"oer, other categorifications of $\mathbb{C}[N]$ were constructed, but relying on categories of different natures. Unlike the \textit{additive categorification} of \cite{GLS} via representations of preprojective algebras, a new kind of categorification called \textit{monoidal categorification} was introduced by Hernandez-Leclerc \cite{HL}. The idea is to identify a given cluster algebra $\A$ with Grothendieck ring of an artinian monoidal category $\C$ via a ring isomorphism required to send the cluster monomials of $\A$ onto classes of simple objects in $\C$. A first class of examples of such categorifications was provided in \cite{HL} for certain unipotent cells of $\mathbb{C}[N]$ associated with Coxeter elements of $W$. Concerning $\mathbb{C}[N]$ itself, it was proved by Hernandez-Leclerc \cite{HL3} that the Grothendieck ring of a certain category of finite-dimensional representations of quantum affine algebras was isomorphic to $\mathbb{C}[N]$. A monoidal categorification of $\mathbb{C}[N]$ (as well as all its unipotent cells) was then constructed in a vast series of works due to  Kang-Kashiwara-Kim-Oh \cite{KKK,KKKO3,KKKOSelecta,KKKO,KK} using the representation theory of quiver Hecke algebras. 
 
  Quiver Hecke algebras (or KLR algebras) were introduced by Khovanov-Lauda \cite{KL} and Rouquier \cite{R} in the purpose of categorifying the negative half of the quantum group $U_q(\mathfrak{g})$. They are a family of $\mathbb{Z}$-graded associative algebras indexed by $Q_{+}$. The category $R-gmod$ of all graded finite-dimensional representations of the $R(\beta), \beta \in Q_{+}$ can be given a monoidal structure. Rouquier \cite{R} and Varagnolo-Vasserot \cite{VV} proved that the set of isomorphism classes of simple objects in $R-gmod$ is in bijection with the  elements of the dual canonical basis of $\mathbb{C}[N]$. In the case of a finite type  Lie algebra $\mathfrak{g}$, these simple objects were classified by McNamara \cite{McNamarafinite} and Kleshchev-Ram \cite{KR} in terms of root partitions (or dominant words) using the combinatorics of good Lyndon words, relying on the works of Leclerc \cite{Leclerc}. This parametrization was shown to be compatible with the monoidal structure of $R-gmod$ (see \cite{Casbi2}) and turns out to be convenient for studying the determinantal modules categorifying the flag minors of $\mathbb{C}[N]$. In \cite{Casbi2}, we used recent results of Kashiwara-Kim \cite{KK} to provide an explicit description in terms of root partitions of the determinantal modules corresponding to the cluster variables of the standard seed $\s^{\mathbf{i}}$ when $\mathbf{i} \in \text{Red}(w_0)$ comes from a total ordering of an index set of simple roots. This description will be very useful in the proof of the second main result of this paper (Theorem~\ref{thminitcond}). 
  
  Kleshchev-Ram \cite{KR} also exhibited a (finite) family of distinguished irreducible representations called \textit{cuspidal representations}, satisfying several good properties. Note that the notion of cuspidality depends on a preliminary choice of total ordering of an index set of simple roots. 
  The understanding of cuspidal modules was the motivation for the construction of \textit{homogeneous representations} over simply-laced type quiver Hecke algebras  by the same authors in \cite{KRhom}. They are a (finite) remarkable family of simple finite-dimensional modules in $R-gmod$ parametrized by the so-called \textit{fully-commutative} elements of $W$. The combinatorics of fully-commutative elements of Weyl groups is very rich and has been studied for a long time by Proctor \cite{P1,P2}, Stembridge \cite{Stem} and Nakada \cite{Nakada} among others. A subfamily of homogeneous modules, called \textit{strongly homogeneous} in the terminology of \cite{KRhom} will be of particular interest for us. It was already observed by Kleshchev and Ram that the dimensions of these modules are given by the Peterson-Proctor hook formula. This formula, introduced in an unpublished work of Peterson-Proctor, was generalized by Nakada \cite{Nakada} in a purely combinatorial context as a colored hook formula. It turns out that these colored hook formulas can be conveniently interpreted in terms of strongly homogeneous representations over quiver Hecke algebras by using certain tools developped in the recent work of Baumann-Kamnitzer-Knutson \cite{BKK}. 
   More details on the structure of homogeneous representations can be found in \cite{FLee1,FLee2}. 

 Mirkovi\'c-Vilonen \cite{MV} exhibited a spectacular equivalence between the category of finite-dimensional representations of a simple simply-connected reductive group $G$ and certain categories of perverse sheaves on the affine Grassmannian $Gr_{G^{\vee}}$ associated to the Langlands dual of $G$. Moreover, the weight subspaces of the highest weights representations of $G$ are interpreted as cohomology spaces of certain sheaves on $Gr_{G^{\vee}}$. This is the \textit{geometric Satake correspondence}. Mirkovi\'c-Vilonen introduced certain interesting closed irreducible subvarieties of $Gr_{G^{\vee}}$ called \textit{MV cycles}, that give rise to interesting bases of the finite-dimensional irreducible representations of $G$ via the geometric Satake correspondence. These bases can be glued together into a basis of $\mathbb{C}[N]$ called the \textit{MV basis}, indexed by certain MV cycles called \textit{stable}. The elements of this basis can also be parametrized in a more combinatorial way using \textit{MV polytopes}. Kamnitzer \cite{KamAdv,Kam} gave an explicit description of these polytopes and showed that they carry a crystal structure in the sense of Kashiwara. These polytopes are known to be the support of certain measures called \textit{Duistermaat-Heckmann measures} introduced by Brion-Procesi \cite{BrionProc}.  
 Considering certain Fourier transforms of the DH measures, Baumann-Kamnitzer-Knutson \cite{BKK}  defined an algebra morphism 
  $$ \overline{D} : \mathbb{C}[N] \longrightarrow \mathbb{C}(\alpha_1, \ldots , \alpha_n) $$
  where $\alpha_1, \ldots , \alpha_n$ are algebraically independent variables representing  the simple roots. One of Baumann-Kamnitzer-Knutson's main results consists in interpreting the image under $\overline{D}$ of an element of the MV basis in terms of certain geometric invariants of the corresponding MV cycle and used such connection as a key step for proving a conjecture of Muthiah \cite{Mu}. These invariants, called \textit{equivariant multiplicities}, are general tools introduced by Joseph \cite{Joseph} and Rossman \cite{Rossmann} and then developped by Brion \cite{Brion}. Given an algebraic variety $X$ together with a torus $T$ acting on $X$, we consider the set $X^T$ of $T$-fixed points in $X$. For any point $p \in X^T$ and any $T$-invariant subvariety $Y \subset X$, the equivariant multiplicity of $Y$ at  $p$ is a rational function denoted $\epsilon_{p}^{T}(Y)$. Brion \cite{Brion} showed that if $p$ is non-degenerate, then $\epsilon_{p}^{T}(Y)$ is always of the form 
  $$ \epsilon_{p}^{T}(Y) = \frac{Q_{p,Y}(\beta_1, \ldots , \beta_m)}{\beta_1 \cdots \beta_m} $$
  where $\beta_1 , \ldots , \beta_m$ are the weights of the action of $T$ on the tangent space of $X$ at $p$ and $Q_{p,Y}$ is some polynomial. These equivariant multiplicities (and in particular the polynomials $Q_{p,Y}$) are difficult to compute in general. The main general property is the following:  
   \begin{equation} \label{eqintro1}
     \text{$p \in Y$ and $Y$ smooth at $p$} \quad \Rightarrow \quad \epsilon_{p}^{T}(Y) = \frac{1}{\beta_1 \cdots \beta_m}  . 
     \end{equation}
Taking $X$ to be the affine Grassmaniann and $Y$ a MV cycle, Baumann-Kamnitzer-Knutson show that the equivariant multiplicity of $Y$ at a certain point of $Y$ coincides with the image under $\overline{D}$ of the MV basis element indexed by $Y$ (see {{\cite[ Corollary 10.6]{BKK}}}). In this situation, the weights involved in the denominator of the previous equality are always positive roots (seen as linear functions in $\alpha_1, \ldots , \alpha_n$). 

 In this paper we will mostly focus on the values of $\overline{D}$ on the elements of the dual canonical basis. Indeed, Nakada's colored hook formula can be straightforwardly interpreted as follows: if $M$ is a strongly homogeneous module in $R-gmod$ and $w$ is the fully-commutative element of $W$ associated to $M$ via the construction of Kleshchev-Ram \cite{KRhom}, then the evaluation of $\overline{D}$ on the isomorphism class of $M$ is of the form 
 \begin{equation} \label{eqintro2}
 \overline{D}([M]) = \frac{1}{\prod_{\beta \in \Phi_{+}^{w}} \beta} . 
  \end{equation}
 This outlines a remarkable similarity between geometric statements requiring certain smoothness conditions of MV cycles and algebraic ones requiring the (strong) homogeneity of certain modules over quiver Hecke algebras.
  On the other hand, the study of various examples (for instance when $\mathfrak{g}$ is of type $A_3$ or $D_4$) suggests the existence of certain coincidences between (prime) strongly homogeneous modules and the determinantal modules categorifying flag minors. This provides a motivation for studying the image under $\overline{D}$ of all flag minors. We propose the following Conjecture, suggesting that $\overline{D}$ surprisingly takes distinguished values similar to~\eqref{eqintro1} and~\eqref{eqintro2} on all flag minors of $\mathbb{C}[N]$, although the corresponding objects in $R-gmod$ may not be strongly homogeneous. In particular, this raises the question of the smoothness of the  MV cycles corresponding to flag minors.
  
    \begin{conjintro} \label{weakconjintro}
       Let $\mathfrak{g}$ be a Lie algebra of finite simply-laced type and let $x$ be a flag minor in $\mathbb{C}[N]$. Then the evaluation of $\overline{D}$ on $x$ is of the form
       $$ \overline{D}(x) =  \frac{1}{\prod_{\beta \in \Phi_{+}} \beta^{n_{\beta}(x)}}  $$
      where $n_{\beta}(x)$ is a nonnegative integer for every positive root $\beta \in \Phi_{+}$. 
       \end{conjintro}
       
   The aim of this paper is to prove the following: 
       
        \begin{thmintro} \label{mainthmintro}
       Assume $\mathfrak{g}$ is of type $A_n, n \geq 1$ or $D_4$. Then Conjecture~\ref{weakconjintro} holds. Moreover, for any standard seed $\s^{\mathbf{i}} = ((x_1, \ldots , x_N) , Q^{\mathbf{i}})$ of $\mathbb{C}[N]$, the polynomials $P_j :=  \left(\overline{D}(x_j) \right)^{-1}$ satisfy the following relations:
         $$ P_j P_{j_{-}(\mathbf{i})} = \beta_j  \prod_{\substack{ l<j<l_{+}(\mathbf{i}) \\ i_l \cdot i_j =-1}} P_l . $$  
         \end{thmintro}
  
 We refer to Section~\ref{fullcom} for the precise definitions of the notations $j_{-}, l_{+}$. 
  The strategy of the proof is the following: we know that the standard seeds are related to each other by certain cluster mutations, corresponding to changes of reduced expressions of $w_0$. Thus one shall first show that the desired statement is preserved under these cluster mutations, so that it only remains to check it for one particular standard seed.
  
   \bigskip
  
  We consider two standard seeds $\s^{\mathbf{i}}$ and $\s^{\mathbf{i}'}$ related by a cluster mutation in the direction $k$. We denote by $x_1, \ldots , x_N$ the cluster variables of $\s^{\mathbf{i}}$ and $x'_k$ the new variable produced by the mutation. We assume that $\overline{D}(x_j)$ is of the form $1/P_j$ for every $1 \leq j \leq N$ and we want to show that $\overline{D}(x'_k)$ is of the form $1/P'_k$. The first main result of this paper consists in exhibiting certain relations between the $P_j$ entirely determined by $\mathbf{i}$ implying that $\overline{D}(x'_k)$ is of the form $1/P'_k$ (where $P'_k$ is a product of positive roots) and proving that these relations are preserved under mutation, i.e. the polynomials $P_1, \ldots , P_{k-1}, P'_k, P_{k+1} , \ldots , P_N$ satisfy the  corresponding relations determined by $\mathbf{i}'$. 
   
  For any product of positive roots $P$ and any positive root $\beta$, we denote by $(\beta ; P)$ the multiplicity of $\beta$ in $P$. 
 
   \begin{thmintro} \label{intro1}
     Let $\mathfrak{g}$ be any simply-laced type simple Lie algebra and let $\mathbf{i}$ and $\mathbf{i}'$ be two reduced expressions of $w_0$.
   Assume that the cluster variables  $x_1, \ldots , x_N$ of the  standard seed $\s^{\mathbf{i}}$ satisfy  the following properties:
     \begin{enumerate}[label=(\Alph*)]
   \item For every $1 \leq j \leq N$, the rational fraction $\overline{D}(x_j)$ is of the form $1/P_j$ where $P_j$ is a product of positive roots. 
   \item For every $1 \leq j \leq N$ one has  
   $$ P_j P_{j_{-}(\mathbf{i})} = \beta_j  \prod_{\substack{ l<j<l_{+}(\mathbf{i}) \\ i_l \cdot i_j =-1}} P_l . $$  
   \item For every $j \in J_{ex}$ and every $1 \leq i \leq N$, one has $ (\beta_i ; P_j) - (\beta_i ; P_{j_{+}(\mathbf{i})}) \leq 1 $. 
  \end{enumerate}
  Then the cluster variables of $\s^{\mathbf{i}'}$ satisfy the analogous properties determined by $\mathbf{i}'$. 
  \end{thmintro}
  
   The Property (B) is rather strong and has interesting consequences, as explained in Remarks~\ref{rkBalgo} and~\ref{rkBdom}. Note that the Property (C) is only relevant in types $D_n$ and $E_6, E_7, E_8$ as in type $A_n$ the polynomials $P_j$ are always square-free. Nonetheless, it is crucial for proving that $\overline{D}(x'_k)$ is of the desired form. 
  
  The second main result of this paper is to exhibit one particular standard seed satisfying the conditions required by the previous Theorem for types $A_n$ and $D_4$. We use Kang-Kashiwara-Kim-Oh's results \cite{KKKO,KK} and more precisely the description of certain determinantal modules in $R-gmod$ in terms of root partitions.
  
   \begin{thmintro} \label{intro2}
 Assume $\mathfrak{g}$ is of type $A_n , n \geq 1$ or $D_4$. Let $\mathbf{i}_{nat}$ denote the reduced expression of $w_0$ corresponding to the natural ordering on the vertices of the Dynkin diagram of $\mathfrak{g}$. Then Properties (A), (B) and (C) hold for the standard seed $\s^{\mathbf{i}_{nat}}$. 
 \end{thmintro}

 As explained above, Theorem~\ref{mainthmintro} follows by combining Theorems~\ref{intro1} and~\ref{intro2}. Moreover, we also get that when $\mathfrak{g}$ is of type $A_n$ or $D_4$, the flag minors of any standard seed of $\mathbb{C}[N]$ satisfy Properties (B) and (C).

 \bigskip
 
 This paper is organized as follows. We begin with some reminders on quiver Hecke algebras and their irreducible finite-dimensional representations (Section~\ref{remindKLR3} and~\ref{remindKR3}). We also recall the constructions of determinantal modules from the works of Geiss-Leclerc-Schr\"oer \cite{GLS} and Kang-Kashiwara-Kim-Oh \cite{KKKO} (Section~\ref{reminddeterm}). In Section~\ref{remindKRHom3}, we gather the main facts on the combinatorics of fully-commutative elements of Weyl groups. We also explain (Section~\ref{homogKLR}) how this combinatorics is related to the representation theory of quiver Hecke algebras via the works of Kleshchev-Ram \cite{KRhom}. Section~\ref{remindMV3} is devoted to the necessary reminders on the theory of Mirkovi\'c-Vilonen cycles and equivariant multiplicities following \cite{BKK,Brion,MV}. In Section~\ref{result3}, we state our main results together with some motivations and explanations about the structure of the proofs. Section~\ref{propag} contains the proof of Theorem~\ref{intro1}. Sections~\ref{initAn} and~\ref{initD4} are respectively devoted to the proofs of Theorem~\ref{intro2} in types $A_n$ and $D_4$. We conclude in Section~\ref{conclu} by discussing several evidences suggesting a cluster-theoretic interpretation of prime strongly homogeneous modules. 
   
  \bigskip
  
    \textit{The author is supported by the European Research Council under the European Union's Framework Programme H2020 with ERC Grant Agreement number 647353 $\allowbreak$ Qaffine.}

\section*{Acknowledgements}
I warmly thank my advisor David Hernandez for his constant encouragements and his kind supervision during these three years of PhD, as well as for many valuable suggestions on a previous version of this work. I would also like to express all my gratitude towards Joel Kamnitzer, Pierre Baumann and Bernard Leclerc, whose explanations and pieces of advice were the starting point of this project. Finally I thank Masaki Kashiwara and Hironori Oya for very insightful and enjoyable conversations.

 \section{Quiver Hecke algebras and their representations}
 
 We begin with some reminders about the representation theory of quiver Hecke algebras and in particular its applications to monoidal categorifications of cluster algebras following the works of Kang-Kashiwara-Kim-Oh \cite{KKK,KKKO,KK}. 
 
  \subsection{Quiver Hecke algebras}
  \label{remindKLR3}
  
  We fix the notations and recall the main properties of quiver Hecke algebras. 
 
 \bigskip
 
   We will always consider a semisimple Lie algebra $\mathfrak{g}$ of finite simply-laced type, with a fixed labeling $I = \{1, \ldots , n \}$ of the set of vertices of the associated Dynkin diagram. We let $\Pi = \{ \alpha_i , i \in  I \}$ denote the set of simple roots, $Q_{+} := \bigoplus_{i \in I} \mathbb{N} \alpha_i$, and $\Phi_{+} \subset Q_{+}$ the set of positive roots. We let $A = (a_{ij})$ denote the Cartan matrix associated to $\mathfrak{g}$ and we consider the symmetric bilinear form $i,j \mapsto i \cdot j$ on $\mathbb{Z}[I]$ defined by $i \cdot j := a_{i,j}$ for every $i,j \in I$.  This induces a symmetric bilinear form $( \cdot , \cdot )$ on $Q_{+}$ defined by $(\alpha_i , \alpha_j) = i \cdot j$ for any $i,j \in I$. 
 We also let $\mathcal{M}$ denote the set of all finite words over the alphabet $I$. For any such word $\mathbf{j} = (j_1, \ldots , j_r)$, the  \textit{weight} of $\mathbf{j}$ is the element of $Q_{+}$ defined as  
$$ \wt(\mathbf{j}) :=  \sum_{i \in I}  \sharp \{k, j_k = i \} \alpha_i  \in Q_{+} . $$
  Quiver Hecke algebras are defined as a family $\{ R(\beta) , \beta \in Q_{+} \}$ of associative $\mathbf{k}$-algebras indexed by $Q_{+}$ (where $\mathbf{k}$ is a fixed algebraically closed field of characteristic different of $2$). For every $\beta := \sum_i a_i \alpha_i \in Q_{+}$, the algebra $R(\beta)$ is generated by three kind of generators: there are polynomial generators $x_1 , \ldots , x_r$, braiding generators $\tau_1 , \ldots , \tau_{r-1}$, (where $r := \sum_i a_i$) and idempotents $e(\mathbf{j}), \mathbf{j} \in \text{Seq}(\beta)$ where $\text{Seq}(\beta)$ is the finite subset of $\mathcal{M}$ given by
 $$ \text{Seq}(\beta) := \{ \mathbf{j} \in \mathcal{M} \mid \wt(\mathbf{j}) = \beta \} . $$
 The idempotent generators commute with the polynomial ones and are orthogonal to each other in the sense that 
 $$ e(\mathbf{j})e(\mathbf{j}') = \delta_{\mathbf{j},\mathbf{j}'} e(\mathbf{j}) . $$
 It is a crucial point that the algebra $R(\beta)$ carries a natural $\mathbb{Z}$-grading given by 
 $$ \deg e(\mathbf{j}) = 0 \quad \deg x_k e(\mathbf{j}) = 2 \quad \deg \tau_l e(\mathbf{j}) = - j_l \cdot j_{l+1} $$
 for every $\mathbf{j} = (j_1, \ldots , j_r) , r \geq 1$. Thus one can consider the category $R(\beta)-gmod$ of finite dimensional graded $R(\beta)$-modules. We also define
 $$ R-gmod := \bigoplus_{\beta} R(\beta)-gmod . $$
 The category $R-gmod$ can be endowed with a structure of a monoidal category via a monoidal product denoted by $\circ$ and constructed as a parabolic induction. Therefore the Grothendieck group $K_0(R-gmod)$ has a ring structure. There is also a grading shift functor in $R-gmod$ which yields a $\mathbb{Z}[q^{\pm 1}]$-module structure on $K_0(R-gmod)$. 
The following results are the main properties of quiver Hecke algebras:

 \begin{thm}[Khovanov-Lauda \cite{KL}, Rouquier \cite{R}] \label{firstthmKLR}
 There is an isomorphism of $\mathbb{Z}[q^{\pm 1}]$-modules
 $$ K_0(R-gmod) \xrightarrow[]{\simeq} \Aqn . $$
 \end{thm}
 
 \begin{thm}[Rouquier \cite{R}, Varagnolo-Vasserot \cite{VV}] \label{secondthmKLR}
 The above isomorphism induces a bijection between the set of classes of simple objects in $R-gmod$ and the dual canonical basis of $\Aqn$. 
 \end{thm}

  \subsection{Irreducible finite-dimensional representations}
  \label{remindKR3}
  
  This subsection is devoted to some reminders about Kleshchev-Ram's parametrization of simple finite-dimensional modules via root partitions (or dominant words) using the combinatorics of good Lyndon words. We also recall the notions of graded character as well as the quantum shuffle product formula. These will be useful in Section~\ref{initD4}. 
 
  \bigskip
   
  For a finite-type simple Lie algebra $\mathfrak{g}$, the simple objects in $R-gmod$ have been classified by Kleshchev-Ram \cite{KR} in terms of root partitions. First fix an arbitrary total order $<$ on $I$. It induces a lexicographic order on $\mathcal{M}$ that we still denote by $<$. Any module in $R-gmod$ can be decomposed as a direct sum of vector spaces 
  $$ M =  \bigoplus_{\mathbf{j} \in \text{Seq}(\beta)} e(\mathbf{j}) \cdot M . $$
  Thus one can consider $\max(M)$ the maximal element $\mathbf{j}$ for $<$ such that $e(\mathbf{j}) \cdot M \neq 0$. Then one can define the \textit{cuspidal} representations in $R-gmod$ in the following way:
  
   \begin{prop}[ {{\cite[Lemma 6.3]{KR}}}]
   For every positive root $\beta \in \Phi_{+}$, there is a unique simple module $S_{\beta}$ in $R(\beta)-gmod$ (up to isomorphism and grading shift) such that
   $$ \max(S_{\beta}) = \min \left( \max(M) \mid \text{$M$ simple in $R(\beta)-gmod$} \right) . $$
  \end{prop}
  
 The module $S_{\beta}$ is called the \textit{cuspidal representation of weight $\beta$}. The map 
   $$\beta \longmapsto \max(S_{\beta}) $$ 
induces a bijection from $\Phi_{+}$ to a finite subset of $\mathcal{M}$ denoted $\mathcal{GL}$ and whose elements are called \textit{good Lyndon words}. The inverse of this bijection is given by 
\begin{equation}  \label{bijGL}
 \begin{array}{ccc}
  \mathcal{GL} & \longrightarrow & \Phi_{+} \\
     \mathbf{j} & \longmapsto &   \wt(\mathbf{j}) .
 \end{array}
 \end{equation}
 There is an algorithm that allows one to compute inductively the elements of $\mathcal{GL}$ (see {{\cite[Section 4.3]{Leclerc}}}). 
The main classification result is the following:

\begin{thm}[Kleshchev-Ram \cite{KR}]
There is a bijection between the set of isomorphism classes of simple objects (up to grading shift) in $R-gmod$ and the set 
$$ \M := \{ {\bf j}_1 \cdots {\bf j}_k \mid {\bf j}_1 , \ldots , {\bf j}_k \in \mathcal{GL} , {\bf j}_1 \geq \cdots \geq {\bf j}_k \}  $$
given by 
 $$ \mu := {\bf j}_1 \cdots {\bf j}_k \in \M \enspace \longmapsto \enspace L(\mu) := \text{hd} \left(L({\bf j}_1) \circ \cdots \circ L({\bf j}_k) \right) $$
  and moreover one has
  $$ \max(L(\mu)) = \mu . $$
  \end{thm}
  
   With the previous notations one has in particular $L(\mathbf{j}) = S_{\wt(\mathbf{j})}$ for every $\mathbf{j} \in \mathcal{GL}$. The elements of $\M$ are called \textit{dominant words}. 
 
  \bigskip
  
 This framework allows one to compute products of modules in $R-gmod$ via the \textit{quantum shuffle product formula}. First, recall that the \textit{graded character} of a module $M$ in $R(\beta)-gmod$ is by definition a formal sum of words with coefficients in $\mathbb{Z}_{\geq 0}[q^{\pm 1}]$ given by
  $$ ch_q(M) := \sum_{\mathbf{j} \in \text{Seq}(\beta)} \dim_q(e(\mathbf{j}) \cdot M) \mathbf{j} . $$
 Recall also that the \textit{quantum shuffle} of two words $\mathbf{j},\mathbf{j}'$ is defined as follows: write $\mathbf{j} = (j_1, \ldots , j_r) , \mathbf{j}' = (j_{r+1} , \ldots , j_{r+s})$.
 Then 
 $$ \mathbf{j} \circ \mathbf{j}' := \sum_{\mathbf{j}'' \in \mathbf{j} \shuffle \mathbf{j}'} q^{-\epsilon_{\mathbf{j}''}} \mathbf{j}'' $$
where
 $$ \mathbf{j} \shuffle \mathbf{j}' := \{ (j_{\sigma^{-1}(1)} , \ldots, j_{\sigma^{-1}(r+s)})  \mid \sigma \in \Sigma_{r+s} , \sigma(1) < \cdots < \sigma(r) ,  \sigma(r+1) < \cdots < \sigma(r+s) \} $$
  and for every $\mathbf{j}'' \in \mathbf{j} \shuffle \mathbf{j}'$, the integer $\epsilon_{\mathbf{j}''}$ is defined by 
$$ \epsilon_{\mathbf{j}''}  := \sum_{\substack{ k \leq r < l \\ \sigma(k) > \sigma(l)}} j_k \cdot j_l \quad \text{for $\sigma$ such that $\mathbf{j}''=(j_{\sigma^{-1}(1)} , \ldots, j_{\sigma^{-1}(r+s)})$.} $$
 Then one can extend by linearity the quantum shuffle product $\circ$ to any (finite) formal sum of words. In particular given two modules $M,N$ in $R-gmod$ one can define $ch_q(M) \circ ch_q(N)$ as 
 $$ ch_q(M) \circ ch_q(N) = \sum_{\substack{\mathbf{j},\mathbf{j}' \\ e(\mathbf{j}) \cdot M \neq 0 \\ e(\mathbf{j}') \cdot N \neq 0}}  \dim_q(e(\mathbf{j}) \cdot M)  \dim_q(e(\mathbf{j}') \cdot N) \quad (\mathbf{j} \circ \mathbf{j}') . $$
 Then one has the quantum shuffle product formula:
 
 \begin{prop}[{{\cite[Lemma 2.20]{KL}}}]  \label{quantumshuffle}
 For every pair of objects $M,N$ in $R-gmod$, one has
 $$ ch_q(M \circ N) = ch_q(M) \circ ch_q(N)  $$
 where the symbol $\circ$ on the left hand side refers to the monoidal product in $R-gmod$. 
 \end{prop}

  \subsection{Determinantal modules}
  \label{reminddeterm}
  
   In this subsection we make some reminders about determinantal modules in $R-gmod$  following \cite{KKKO}, adapting a former construction due to Geiss-Leclerc-Schr\"oer \cite{GLS}. These constructions are also valid  more generally in any $\Cw , w \in W$ but here we will use them only for $R-gmod = \C_{w_0}$ where $w_0$ stands for the longest element of the Weyl group $W$ associated to $\mathfrak{g}$. 
   
   \bigskip
   
  Geiss-Leclerc-Schr\"oer \cite{GLS} constructed a family of distinguished seeds for the cluster structures of $\mathbb{C}[N]$. Their construction actually involves the quantum cluster structures on the quantum coordinate ring $\Aqn$ but here we will always work in the classical setting.  Geiss-Leclerc-Schr\"oer consider certain elements $D(u \lambda , v \lambda)$ of $\mathbb{C}[N]$ called \textit{(quantum) unipotent minors} parametrized by triples $(\lambda,u,v)$ consisting of a dominant weight $\lambda$ and a couple $(u,v)$ of elements of $W$. These unipotent minors always belong to the dual canonical basis of $\mathbb{C}[N]$ when they are not zero (see for instance {{\cite[Lemma 9.1.1]{KKKO}}}).  These elements satisfy certain remarkable relations called \textit{determinantal identities} (see {{\cite[Proposition 5.4]{GLS}}}). Note that such identities already appear in the work of Fomin-Zelevinsky \cite{FZJAMS}.  A particularly interesting family of determinantal identities are obtained by considering certain special unipotent minors, obtained by taking $\lambda$ to be a fundamental weight $\omega_i , i \in I$ together with Weyl groups elements  of the form $u=s_{i_1} \cdots s_{i_k}$ and $v=s_{j_1} \cdots s_{j_l}$ such that $i_k = j_l = i$ and $i_p = j_p$ if $p \leq \min(k,l)$. 
  Geiss-Leclerc-Schr\"oer \cite{GLS} prove that the determinantal identities relating these unipotent minors can be interpreted as exchange relations associated to cluster mutations in $\mathbb{C}[N]$.

  Geiss-Leclerc-Schr\"oer construct a family of seeds $\{ \s^{\mathbf{i}} , \mathbf{i} \in \text{Red}(w_0) \}$ in $\mathbb{C}[N]$, called \textit{standard seeds}, indexed by the reduced expressions of $w_0$. For each $\mathbf{i} = (i_1, \ldots , i_N) \in \text{Red}(w_0)$, the quiver of the seed $\s^{\mathbf{i}}$ can be constructed from $\mathbf{i}$ as explained in Section~\ref{propag} below. The cluster variables of the seed $\s^{\mathbf{i}}$ are the unipotent minors
  $$ D(s_{i_1} \cdots s_{i_k} \omega_{i_k} , \omega_{i_k}) \quad 1 \leq k \leq N .  $$ These minors  are called \textit{(quantum) flag minors}. The determinantal identities between flag minors is crucial in the study of the cluster  structure of $\mathbb{C}[N]$.  Note that the rank of this cluster structure  (i.e. the number of cluster variables in each seed of $\mathbb{C}[N]$) is equal to the length of $w_0$ or equivalently to the number of positive roots of $\mathfrak{g}$. 
  The results of \cite{GLS}  mainly rely on additive categorification techniques using representations of the preprojective algebra. Kang-Kashiwara-Kim-Oh \cite{KKKO} adapted this construction to the monoidal setting  by lifting the non zero unipotent minors of $\mathbb{C}[N]$ to $R-gmod$ via the isomorphism of Theorem~\ref{firstthmKLR}. The modules obtained this way are unique (up to isomorphism and grading shift) and are called \textit{determinantal modules}. As the non zero unipotent minors are always elements of the dual canonical basis, it follows from  Theorem~\ref{secondthmKLR} that the determinantal modules are simple. They are also known to be \textit{real} in the sense of Hernandez-Lelcerc \cite{HL} (see  {{\cite[Lemma 10.2.2]{KKKO}}}). 
  
  In particular, for each $\mathbf{i} = (i_1 , \ldots , i_N) \in \text{Red}(w_0)$ and for every $1 \leq k \leq N$,  we denote by $M_k^{\mathbf{i}}$ the determinantal module defined by (up to isomorphism and shift)
  $$ [M_k^{\mathbf{i}}] =  D(s_{i_1} \cdots s_{i_k} \omega_{i_k} , \omega_{i_k}) $$
  where $[M]$ denotes the class of $M$ in $K_0(R-gmod)$.
One of the main results due to Kang-Kashiwara-Kim-Oh ({{\cite[Theorem 11.2.2]{KKKO}}})  is to prove that the datum the modules $M_1^{\mathbf{i}}, \ldots  M_N^{\mathbf{i}}$ together with the quiver $Q^{\mathbf{i}}$ forms an monoidal seed admitting successive monoidal mutations in any exchangeable directions (in the sense of {{\cite[Definitions 6.2.1 and 6.2.3]{KKKO}}}). This allows them to prove that $R-gmod$ is a \textit{monoidal categorification} in the sense of Hernandez-Leclerc \cite{HL} of the cluster structure of $\mathbb{C}[N]$ (as well as analogous statements for each of the unipotent cells $\mathbb{C}[N(w)] , w \in W$).
  $$
  \xymatrixcolsep{0.3pc} \xymatrix{ R-gmod \ar[d] & \ni & (M_1^{\mathbf{i}}, \ldots  M_N^{\mathbf{i}}) \ar@{|->}[d] & \quad \text{determinantal modules associated to $\mathbf{i}$} \\
      \mathbb{C}[N] & \ni & (x_1^{\mathbf{i}}, \ldots, x_N^{\mathbf{i}}) & \quad \text{flag minors of the seed $\s^{\mathbf{i}}$} .}
      $$ 
  By {{\cite[Proposition 10.2.4]{KKKO}}}, the determinantal modules can be constructed inductively by successive applications of the induction functors categorifying the crystal structure of the dual canonical basis of $\mathbb{C}[N]$. In \cite{Casbi2}, we used Kashiwara-Kim's results \cite{KK} to provide a combinatorial description of the determinantal modules appearing in the seeds 
  $$  \s^{\mathbf{i}}  , \text{$\mathbf{i}$ corresponding to a total ordering on simple roots} . $$
  This description uses Kleshchev-Ram's parametrization via dominant words recalled in the previous subsection. Fix a labeling $I$ of the set of simple roots of $\mathfrak{g}$ and let $<$ be any arbitrary order on $I$. We still denote by $<$ the induced lexicographic order $<$ on the set $\mathcal{GL}$  of good Lyndon words (see Section~\ref{remindKLR3}) . It yields an order on the set of positive roots $\Phi_{+}$ via the bijection~\eqref{bijGL}. By the results of Rosso \cite{Rosso} this order turns out to be  a convex order (see {{\cite[Proposition 26]{Leclerc}}}).
  Let ${\bf i}_{<}$ denote the corresponding reduced expression of $w_0$
   and consider the seed $\mathcal{S}^{\mathbf{i}_{<}}$ in $\mathbb{C}[N]$. 
  Note that different orderings on $I$ may give the same seed. Moreover certain of the seeds $\mathcal{S}^{{\bf i}}$ may not come from any ordering $<$. 
  
  \begin{ex}
  Let $\mathfrak{g}$ be of type $A_3$. Then there are 14 seeds in $\mathbb{C}[N]$, 6 of them being of the form $\mathcal{S}^{{\bf i}}$. Among these, 5 of them are of the form $\mathcal{S}^{{\bf i}_{<}}$. 
  \end{ex}
  
  We write the reduced expression $\mathbf{i}_{<}$ of $w_0$ as $\mathbf{i}_{<} = (i_1, \ldots , i_N)$ (with $N := l(w_0)$). The determinantal modules of the seed $\mathcal{S}^{\mathbf{i}_{<}}$ can be described in terms of dominant words as follows 
   
    \begin{thm}[{{\cite[Theorem 3.7]{Casbi2}}}] \label{thmcasbi2}
 Let $(x_1, \ldots , x_N)$ denote the cluster variables of the seed $\mathcal{S}^{\mathbf{i}_{<}}$ and let $\mu_k \in \M$ denote the dominant word such that $x_k=[L(\mu_k)]$ for every $1 \leq k \leq N$. Write
 $$ \mu_k := ({\bf j}_{N})^{c_N} \cdots  ({\bf j}_{1})^{c_1} . $$
 Then the tuple  $(c_1, \ldots , c_N)$ is given by 
 $$ c_j =  \begin{cases}
        1  \quad  & \text{if $j \leq k$ and $i_j=i_k$} \\
        0 \quad &\text{otherwise.}
        \end{cases} $$
 \end{thm}

\section{Homogeneous modules over quiver Hecke algebras}
\label{remindKRHom3}

 This section is devoted to some reminders about Kleshchev-Ram's construction of homogeneous simple modules over quiver Hecke algebras of finite simply-laced type. These constitute a finite class of simple objects in $R-gmod$ characterized by their being concentrated in a single degree (which explains the terminology \textit{homogeneous}). They are classified by the rich combinatorics of fully-commutative elements of Weyl groups. We begin by recalling  several known  combinatorial properties of these elements after the works of Peterson-Proctor \cite{P1,P2}, Stembridge \cite{Stem} and Nakada \cite{Nakada}. 

 \subsection{Combinatorics of fully-commutative elements of Weyl groups}
 \label{fullcom}
 
 We recall known combinatorial facts about several remarkable classes of elements of Weyl groups, namely the fully-commutative, minuscule and dominant minuscule elements. Throughout this section $W$ will stand for the Weyl group associated to a Lie algebra $\mathfrak{g}$ of finite Dynkin type (but not necessarily simply-laced). We let $\{s_i , i \in I \}$ denote the set of the simple reflections of $W$, where $I$ is a finite index set.  We also let $P$ (resp. $P^{\vee}$) denote the weight lattice (resp. the coweight lattice) of $\mathfrak{g}$. We denote by $\alpha_i, i \in I$ (resp. $\alpha_i^{\vee}, i \in I$) the simple roots (resp. the simple coroots) and we let $\langle \cdot , \cdot \rangle$ denote  the bilinear form on $P^{\vee} \times P$ such that $\langle \alpha_i^{\vee} , \alpha_j \rangle = i \cdot j$ for every $i,j \in I$. 
 
  \bigskip
  
 The following definition is due to Stembridge \cite{Stem} generalizing definitions  introduced by Peterson-Proctor relying on previous works by Proctor \cite{P1,P2}. 
 
  \begin{deftn}  \label{definitionfullcom}
  An element $w \in W$ is said to be fully-commutative if for any pair $(i,j) \in I$ such that $i \cdot j \neq 0$, there is no reduced expression of $w$ containing a subword of the form $(i,j,i,j, \ldots )$ of length $m$, where $m$ is the order of $s_is_j$ in $W$. 
  \end{deftn} 
  
  Note that if an element $w \in W$ admits a reduced expression satisfying this property, then it is also the case for every reduced expression of $w$. In the case of a simply-laced Weyl group, the notion of fully-commutative is equivalent to requiring that no reduced expression of $w$ should contain a subword of braid form $(i,j,i)$ with $i \cdot j = -1$. In other words, all the reduced expressions of $w$ can be recovered from one given reduced expression by only performing changes of the form $(\ldots , k , l , \ldots ) \rightarrow (\ldots , l , k , \ldots )$ for $k,l \in I$ such that $k \cdot l = 0$. 
  
  \bigskip

  One now defines minuscule and dominant minuscule elements of $W$.
  
   \begin{deftn}  \label{defminuscule}
   An element $w \in W$ is said to be minuscule (resp. dominant minuscule) if there exists an integral weight  $\lambda \in P$ (resp. a dominant  integral weight $\lambda \in P^{+}$) such that for any reduced expression $(j_1, \ldots , j_N)$ of $w$, one has 
$$ \langle \alpha_{j_{k}}^{\vee} , s_{j_{k+1}} \cdots s_{j_N} \lambda \rangle  = 1 $$
for every $1 \leq k \leq N$. 
\end{deftn}

 As before one can replace "for any reduced expression" by "there exists a reduced expression" in this definition (see {{\cite[Propositions 2.1]{Stem}}}). Stembridge \cite{Stem} gave the following very useful classification of minuscule and dominant minuscule elements of $W$. We use the following notation from \cite{GLS}: if $ \mathbf{j} = (j_1, \ldots , j_N)$ is a reduced expression of an element $w \in W$, then for every $1 \leq k \leq N$ we set 
  $$ k_{+} := \min \left( \{ k < l \leq N  \mid j_l = j_k \} \cup \{ N+1 \} \right), $$
  $$ k_{-} := \max \left( \{ l < k \leq N  \mid j_l = j_k \} \cup \{ 0 \} \right), $$
  In other words $k_{+}$ (resp. $k_{-}$) is the position in $\mathbf{i}$ of the next (resp. previous) occurrence of the letter $i_k$ and one has $k_{+}=N+1$ (resp. $k_{-}=0$) if and only if $k$ is the position of the last (resp. first) occurrence of the letter $j_k$ in $\mathbf{j}$. 
 
  \begin{thm}[{{\cite[Propositions 2.3, 2.5]{Stem}}}]  \label{Stembridge}
 Let $w \in W$ and fix a reduced expression $(j_1, \ldots , j_N)$ of $w$. 
  \begin{itemize}
  \item The element $w$ is minuscule if and only if for every $1 \leq k \leq N$, one has 
 $$ k_{+}  \leq N \quad \Rightarrow \quad \sum_{k<l<k_{+}} j_k \cdot j_l = -2 . $$   
   \item The element $w$ is dominant minuscule if and only if it is minuscule and in addition, one has for every $1 \leq k \leq N$,
$$ k_{+} = N+1 \quad \Rightarrow \quad \sum_{l>k} j_k \cdot j_l \geq -1 . $$
  \end{itemize}
   \end{thm}

    \begin{rk}
     \begin{enumerate}
   \item  It is proven by Stembridge \cite{Stem} that minuscule (and a fortiori dominant minuscule) elements are fully-commutative. 
     \item  If $\mathfrak{g}$ is of type $A_n, n \geq 1$ then one can show that every fully-commutative element is in fact minuscule so that these two notions coincide. But in other types, the notions of fully-commutative, minuscule and dominant minuscule elements are distinct. 
     \end{enumerate}
  \end{rk} 
  
   \begin{ex} 
   In type $A_3$, the elements $s_1s_2s_3$ and $s_2s_1s_3s_2$ are dominant minuscule. The element $s_2s_3s_1$ is minuscule but not dominant minuscule. The element $s_3s_2s_1s_2$ is not fully-commutative. 
   
    In type $D_4$ with $3$ being the trivalent node of the Dynkin diagram, the element $s_3s_1s_2s_4s_3$ is fully-commutative but is not minuscule (and a fortiori not dominant minuscule). 
    \end{ex}
    
    We will  respectively denote by $\mathcal{FC}, \mathcal{M}in, \mathcal{M}in^{+}$ the sets of fully-commutative, minuscule, and dominant minuscule elements of $W$. 
 
 \subsection{Colored hook formulas}
 
 In this paragraph we present the most remarkable properties of dominant minuscule elements of Weyl groups, namely the Peterson-Proctor hook formula as well a generalized version established by Nakada \cite{Nakada} called \textit{colored} Peterson-Proctor hook formula. 
 
  We still consider the Weyl group $W$ of a finite type Lie algebra $\mathfrak{g}$ and we fix $w$ a dominant minuscule element of $W$. We also let $N := l(w)$ denote the length of $w$ and $\text{Red}(w)$ the set of all reduced expressions of $w$. Peterson-Proctor proved a formula for the cardinality of $\text{Red}(w)$. We let 
  $$ \Phi_{+}^w := \Phi_{+} \cap w \Phi_{-} $$
  denote the set of positive roots associated to $w$. This set is of cardinality $l(w)$. The elements of $\Phi_{+}^w$ are given in the following way: choose a reduced expression $(j_1, \ldots, j_N)$ of $w$. Then $\Phi_{+}^w = \{ \beta_1 , \ldots , \beta_N \}$ with 
  $$ \beta_k := s_{j_1} \cdots s_{j_{k-1}}(\alpha_{j_{k}}) $$ 
 for every $1 \leq k \leq N$. 
  Recall that this set $\Phi_{+}^w$ does not depend on any choice of reduced expression of $w$ and that $\Phi_{+}^{w} \neq  \Phi_{+}^{w'}$ if $w \neq w'$. The Peterson-Proctor hook formula can be written as follows:
  
   \begin{prop}[Peterson-Proctor] \label{Peterson-Proctor}
  Let $w$ be a dominant minuscule element of $W$. Then the number of reduced expressions of $w$ is given by
  \begin{equation} \label{eqPeterProc}
   \sharp \text{Red}(w) = \frac{l(w)!}{\prod_{\beta \in \Phi_{+}^{w}} \hgt(\beta)}.
     \end{equation}
  \end{prop} 
  
  Nakada \cite{Nakada} proved a generalisation of this formula. Recall the following notations from {{\cite[Section 6]{Casbi2}}}: we consider the simple roots $\alpha_1 , \ldots , \alpha_n$ of $\mathfrak{g}$ as independent formal variables and for every positive root $\beta = a_1 \alpha_1 + \cdots + a_n \alpha_n \in \Phi_{+}$ with $a_1, \ldots , a_n \in \mathbb{Z}_{\geq 0}$,  we consider the rational function 
  $$ \frac{1}{\beta} := \frac{1}{a_1 \alpha_1 + \cdots + a_n \alpha_n} \in \mathbb{R}(\alpha_1, \ldots , \alpha_n) . $$
  Nakada's colored hook formula can be written as follows: 
  
  \begin{thm}[{{\cite[Corollary 7.2]{Nakada}}}] \label{HookNakada}
 Let $w$ be a dominant minuscule element of $W$. Then the following equality holds in $\mathbb{R}(\alpha_1, \ldots , \alpha_n)$:
  \begin{equation} \label{eqNakada}
   \prod_{\beta \in \Phi_{+}^{w}} \frac{1}{\beta} 
 =  \sum_{(j_1, \ldots , j_N) \in \text{Red}(w)} \frac{1}{\alpha_{j_1}} \frac{1}{\alpha_{j_1} + \alpha_{j_2}} \cdots \frac{1}{\alpha_{j_1} + \alpha_{j_2} + \cdots + \alpha_{j_N}}  
   \end{equation} 
 \end{thm}
 
 Specializing all the variables $\alpha_j$ to $1$ we recover the Peterson-Proctor hook formula~\eqref{eqPeterProc}. 
 
 \subsection{Homogeneous representations in $R-gmod$}
 \label{homogKLR}
 
  We now recall a construction due to Kleshchev-Ram \cite{KRhom} of certain distinguished irreducible finite dimensional representations over quiver Hecke algebras. We assume that $\mathfrak{g}$ is of finite simply-laced type. Consider the category $R-gmod$ of graded finite dimensional modules over the  quiver Hecke algebras associated to $\mathfrak{g}$. One of Kleshchev-Ram's motivations was the study of cuspidal modules in $R-gmod$, i.e. certain simple modules playing the role of fundamental representations in Lie theory. This led them to introduce a larger (but still finite) class of simple objects in $R-gmod$, called homogeneous modules. 
  
  \begin{deftn}[Kleshchev-Ram]
  A module in $R-gmod$ is called \textit{homogeneous} if it is concentrated in a single degree with respect to the natural grading of quiver Hecke algebras. 
   \end{deftn}
  
  We let $\mathcal{H}om$ denote the set of simple homogeneous modules in $R-gmod$ up to isomorphism and grading shift. 
   The following statement shows that the elements of $\mathcal{H}om$  are parametrized in a natural way by fully-commutative elements of the Weyl group of $\mathfrak{g}$. 
   
    \begin{thm}[{{\cite[Theorem 3.6]{KRhom}}}] \label{thmKRhom}
 There is a bijection 
 $$ \begin{array}{ccc}
     \mathcal{FC} & \longrightarrow & \mathcal{H}om \\
             w   & \longmapsto  &  S(w) 
 \end{array} $$
 between the set of fully-commutative elements of $W$ and the set of isomorphism classes of simple homogeneous modules in $R-gmod$. 
  The  homogeneous module $S(w)$ admits the following decomposition into weight subspaces:
 $$ S(w) = \bigoplus_{(j_1, \ldots , j_N) \in \text{Red}(w)} \mathbb{C} e(j_1, \ldots, j_N) . $$
 In other words the weight spaces of $S(w)$ are all one-dimensional and they are in bijection with the set of reduced expressions of $w$. 
 
 The image of $\mathcal{M}in^{+}$ via the above bijection is a subfamily of homogeneous modules called  \textit{strongly homogeneous} modules. 
  \end{thm}
  
   \begin{rk}
    It follows from the previous Theorem together with Proposition~\ref{Peterson-Proctor} that if $w$ is dominant minuscule, then the strongly homoheneous module $S(w)$ has dimension 
  $$ \dim S(w) = \frac{l(w)!}{\prod_{\beta \in \Phi_{+}^{w}} \hgt(\beta)} . $$
  \end{rk}
 
 \begin{rk}
 In type $A_n$, the cuspidal modules for arbitrary orderings on $\Phi_{+}$ are always homogeneous. In fact, if $M$ is a simple module of multiplicity-free weight, i.e. $M$ is a simple $R(\beta)$-module with $\beta \in Q_{+}$ of the form $\epsilon_1 \alpha_1 + \cdots + \epsilon_n \alpha_n$ with $\epsilon_1 , \ldots , \epsilon_n \in \{0,1\}^n$, then $M$ is always homogeneous. In the particular case of the natural ordering on the vertices of the type $A_n$ Dynkin diagram, i.e. $1 < 2 < \cdots < n$, the cuspidal modules in $R-gmod$ turn out to be strongly homogeneous. 
 
  In type $D_n$,  the cuspidal modules for the natural ordering  are always homogeneous (see {{\cite[Section 8.2]{KR}}}).
  \end{rk}
 
 \section{Mirkovi\'c-Vilonen cycles and equivariant multiplicities}
 \label{remindMV3}
 
  Throughout this section $G$ denotes a  semisimple simply-connected group, $N$ the unipotent radical of a Borel subgroup $B$ of $G$, and $T$ a maximal torus in $B$. We let $\mathfrak{g}, \mathfrak{n}, \mathfrak{b}, \mathfrak{t}$ denote their respective Lie algebras. We fix a labeling $I = \{1 , \ldots , n \}$ of the vertices of the Dynkin diagram of $\mathfrak{g}$ and we let $\alpha_1, \ldots , \alpha_n$ denote the corresponding simple roots. We also let $P$ denote the weight lattice and $W$ the Weyl group of $\mathfrak{g}$. 
 
  \subsection{Geometric Satake correspondence and MV basis}
  
   Here we fix notations and recall the main features of the geometric Satake correspondence following \cite{MV,BKK}. We also refer to \cite{Zhu} for more details about the geometry of the affine Grassmannian and the geometric Satake correspondence. Mirkovi\'c-Vilonen \cite{MV} discovered an intriguing connection between the category of  finite-dimensional representations of a complex reductive algebraic group $G$ on the one hand and a certain category of perverse sheaves on the affine Grassmannian $Gr_{G^{\vee}}$ associated to the Langlands dual of $G$ on the other hand. They exhibited a functor relating these two categories and proved that it satisfies several remarkable properties. In particular, the weight spaces of irreducible representations of $G$ can be interpreted as the cohomology  spaces of certain subvarieties of $Gr_{G^{\vee}}$ whose irreducible components are called \textit{Mirkovi\'c-Vilonen cycles}. One can use these to define an interesting basis of the algebra $\mathbb{C}[N]$ of functions on the unipotent radical $N$ of $G$ called the \textit{MV basis}. The works of Kamnitzer \cite{Kam}, Anderson \cite{And} and Baumann-Kamnitzer \cite{BK} show that this basis carries a natural crystal structure and has a combinatorial parametrization in terms of \textit{MV polytopes}. 
   
    \bigskip
   
   We consider the  algebraically closed field $\mathbb{C}$ and we set 
   $$ \mathcal{O} := \mathbb{C}[[t]] \quad \text{and} \quad \mathcal{K} := \mathbb{C}((t)) . $$
   We denote by $G^{\vee}$ the Langlands dual of $G$. The affine Grassmannian $Gr_{G^{\vee}}$ can be defined as 
   $$ Gr_{G^{\vee}} :=  G^{\vee}(\mathcal{K}) / G^{\vee}(\mathcal{O}) . $$
   There is a natural action of $T^{\vee}(\mathbb{C})$ ($T^{\vee}$ is a maximal torus in $G^{\vee}$) on $Gr_{G^{\vee}}$ whose locus of fixed points is given by a collection $\{ L_{\mu} , \mu \in P \}$ of points in $Gr_{G^{\vee}}$ indexed by the weight lattice of $G$. There is also a $G^{\vee}(\mathcal{O})$-action on $Gr_{G^{\vee}}$ and for a dominant weight $\lambda \in P^{+}$, we let $Gr_{G^{\vee}}^{\lambda}$ denote the orbit of $L_{\lambda}$ under this action. The affine Grassmannian can be decomposed as 
   $$ Gr_{G^{\vee}} = \bigsqcup_{\lambda \in P^{+}} Gr_{G^{\vee}}^{\lambda} . $$
   Moreover, for every $\lambda \in P^{+}$ one has 
   $$ \overline{Gr_{G^{\vee}}^{\lambda}} = \bigcup_{\substack{\mu \in P^{+} \\ \mu \leq \lambda}} Gr_{G^{\vee}}^{\mu} . $$
   where $\leq$ is the natural partial ordering on $P$ defined by $\mu \leq \lambda \Leftrightarrow \lambda - \mu \in Q_{+}$ (recall the notation $Q_{+}$ from Section~\ref{remindKLR3}).
   Let us now briefly recall the definition of Mirkovi\'c-Vilonen cycles. The regular dominant weight $\rho := (\sum_{\alpha \in \Phi_{+}} \alpha)/2$ induces a homomorphism $\mathbb{C}^{\times} \rightarrow T^{\vee}(\mathbb{C})$ and thus yields a $\mathbb{C}^{\times}$-action on $Gr_{G^{\vee}}$. The points $L_{\mu}, \mu \in P$ are exactly the fixed points of this $\mathbb{C}^{\times}$-action. Denoting respectively $S^{\mu}$ and $S_{-}^{\mu}$ the attractive and repulsive set of $L_{\mu}$, we have the Bialynicki-Birula decomposition
 $$ Gr_{G^{\vee}} = \bigcup_{\mu \in P} S^{\mu} =  \bigcup_{\mu \in P} S_{-}^{\mu} . $$
For every $\lambda \in P_{+}$ and  $\mu \in P$, a MV-cycle of type $\lambda$ and weight $\mu$ is by definition an irreducible component of the closed subvariety of $Gr_{G^{\vee}}$ given by
$$ \overline{Gr_{G^{\vee}}^{\lambda} \cap S_{-}^{\mu}} . $$
 One denotes $\mathcal{Z}(\lambda)_{\mu}$ the set of all MV cycles of type $\lambda$ and weight $\mu$ and 
 $$\mathcal{Z}(\lambda) := \bigcup_{\mu \in P} \mathcal{Z}(\lambda)_{\mu} .  $$ 
The geometric Satake correspondence \cite{MV} implies that the MV cycles carry a lot of information about the representation theory of $G$. For instance, the set of fundamental classes in $H_{\bullet}(\overline{Gr_{G^{\vee}}^{\lambda} \cap S_{-}^{\mu}})$ of all MV cycles of type $\lambda$ and weight $\mu$ forms a basis of the subspace $L(\lambda)_{\mu}$  of weight $\mu$ in the irreducible representation $L(\lambda)$ of $G$ of highest weight $\lambda$. In particular one has 
$$ \dim_{\mathbb{C}} L(\lambda)_{\mu} = \sharp \mathcal{Z}(\lambda)_{\mu} . $$
 Gathering these bases together we get that the set 
 $$ \{ [Z]  \mid Z \in \mathcal{Z}(\lambda) \} $$
 forms a basis of $L(\lambda)$ called the (upper) MV basis of $L(\lambda)$. 
Then there is a way to glue these bases for all $\lambda \in P^{+}$ to get a basis of $\mathbb{C}[N]$ called the MV basis of $\mathbb{C}[N]$ (see {{\cite[Section 6.1]{BKK}}} for more details). This basis is denoted $\{ b_{Z} , Z \in \mathcal{Z}(\infty) \}$  where the parametrizing set $\mathcal{Z}(\infty)$ is given by certain MV cycles called \textit{stable MV cycles}, whose weights belong to $-Q_{+}$.

   \subsection{Equivariant multiplicities of MV cycles}
  
 In this subsection we recall several notions introduced by Baumann-Kamnitzer-Knutson \cite{BKK}. One of their main motivations was Muthiah's conjecture \cite{Mu} stating the $W$-equivariance of a certain map  $L(\lambda) \longrightarrow \mathbb{C}(\alpha_1, \ldots , \alpha_n)$ where $L(\lambda)$ is the irreducible representation of highest weight $\lambda$ of $G$. This map is defined using geometric tools called \textit{equivariant multiplicities} developped by Brion \cite{Brion} relying on former constructions due to Joseph \cite{Joseph} and Rossmann \cite{Rossmann}. The proof of \cite{BKK} crucially involves the geometric Satake correspondence together with a formula for Duistermaat-Heckmann measures proved by Knutson  \cite{Knu}. Here we will mostly focus on equivariant multiplicities  of elements of various bases of $\mathbb{C}[N]$ such as the Mirkovi\'c-Vilonen basis or the dual canonical basis. 
 
  \bigskip
  
  The notion of \textit{equivariant multiplicity} of a closed projective scheme has been introduced by Brion \cite{Brion}. Given such a scheme $X$ together with an action of a torus $T$ on $X$, we let $X^{T}$ denote the set of fixed points of this action and $H_{\bullet}^{T}(X)$ denote the $T$-equivariant homology of $X$. We assume all the fixed points are non-degenerate, i.e. for each $p \in X^T$ the weights of the action of $T$ on the tangent space of $X$ at $p$ are non zero. It follows from Brion's results \cite{Brion} that the set of homology classes of the points in the fixed locus $X^{T}$ actually forms a $\mathbb{C}(\text{Lie} T)$-basis of $S^{-1}H_{\bullet}^{T}(X)$ for a certain multiplicative subset $S$ of $H_{\bullet}^{T}(X)$ (see also {{\cite[Section 9.2]{BKK}}}). Therefore one can decompose the class of $X$ on this basis as 
  $$ [X] = \sum_{p \in X^{T}} \epsilon_{p}^{T}(X) [\{p\}] . $$
  The coefficient $\epsilon_{p}^{T}(X) \in \mathbb{C}(\text{Lie}T)$ is called the equivariant multiplicity of $X$ at $p$. Let us state an important property of these equivariant multiplicities, that can be obtained as a consequence of {{\cite[Theorem 4.2]{Brion}}}. 
  
   \begin{prop}[{{\cite[Theorem 4.2]{Brion}}}] \label{propBrion}
   Let $p \in X^T$ non-degenerate and let $\beta_1 , \ldots , \beta_m$ denote the weights  of the action of $T$ on the tangent space of $X$ at $p$. Then for any closed $T$-invariant subvariety $Y \subset X$ containing $p$, one has 
   $$ \text{$Y$ is smooth at $p$} \quad  \Rightarrow \quad \epsilon_{p}^{T}(Y) = \frac{1}{ \beta_{i_1} \cdots \beta_{i_r}} $$
   where $r := \dim(Y)$ and $\beta_{i_1} , \ldots , \beta_{i_r}$ denote the weights of the action of $T$ on $T_x Y \subset T_x X$.  
   \end{prop}

    \begin{rk} 
    \begin{enumerate} 
    \item  Note that by {{\cite[Theorem 4.2 (iii)]{Brion}}}  this implication is actually an equivalence in the special case where $Y$ is $X$ itself. 
    \item If $Y$ does not contain $p$, then $\epsilon_{p}^{T}(Y) = 0$ (see {{\cite[Theorem 4.2 (i)]{Brion}}}). 
    \item Brion \cite{Brion} also shows that this notion of equivariant multiplicity actually coincides with a definition due to Rossmann \cite{Rossmann}. More precisely, he proves that when $X$ is smooth at the point $p$,  Rossmann's equivariant multiplicity of $Y$ at $p$ is equal to $(\beta_1 \cdots \beta_m) \cdot  \epsilon_{p}^{T}(Y)$ (see {{\cite[Theorem 4.5]{Brion}}}). 
     \end{enumerate}
     \end{rk}
     
        \smallskip
 
 Baumann-Kamnitzer-Knutson \cite{BKK} used this notion of equivariant multiplicity in the study of the MV basis of $\mathbb{C}[N]$ via Duistermaat-Heckmann measures. These measures were already known to be supported on the MV polytopes. One of the main results of \cite{BKK} is to provide a formula allowing to relate these measures to the equivariant multiplicities of MV cycles. More precisely, with the notations of Proposition~\ref{propBrion}, we consider $X := Gr_{G^{\vee}}$ the affine Grassmaniann  associated to a semisimple algebraic reductive group $G$, together with the action of the torus $T^{\vee}(\mathbb{C})$.  As recalled in the previous section, the set of fixed points of this action is $\{ L_{\mu} , \mu \in P \}$. For each $\mu \in P$, the point $L_{\mu}$ is known to belong to any MV cycle of weight $\mu$ and it is known that $L_{\mu}$ is non-degenerate. The MV cycles are closed irreducible subvarieties of $Gr_{G^{\vee}}$ and are  invariant under the action of $T^{\vee}$ and thus will play the role of $Y$ in Proposition~\ref{propBrion}. As the present paper mostly deals with the simply-laced case, we will make the identification $\text{Lie} T^{\vee}(\mathbb{C}) \simeq \text{Lie} T(\mathbb{C})$ and hence identify the field $\mathbb{C}(\text{Lie}  T^{\vee}(\mathbb{C}))$ with $\mathbb{C}(\alpha_1, \ldots , \alpha_n)$ where $\alpha_1, \ldots , \alpha_n$ are indeterminates in bijection with the simple roots of $G$. We refer to {{\cite[Corollary 10.6]{BKK}}} for more precise statements.
 
  Let us describe the formulas proved in \cite{BKK}. It is known (see for instance \cite{GLSAENS,GLS}) that the algebra $\mathbb{C}[N]$ can be identified with the dual (as a Hopf algebra) of $U(\mathfrak{n})$. We will denote by $(\cdot , \cdot)$ this duality. Choose a root vector $e_i \in \mathfrak{n}$ of weight $\alpha_i$ for each $i \in I$. For any $N \geq 1$ and any $N$-tuple $\mathbf{j} = (j_1, \ldots , j_N)$ of elements of $I$, we set $e_{\mathbf{j}} := e_{j_1} \cdots e_{j_N} \in U(\mathfrak{n})$. We also define the following rational fraction in $\alpha_1, \ldots , \alpha_n$ following \cite{BKK}:
 $$ \overline{D}_{\mathbf{j}} := \frac{1}{\alpha_{j_1} (\alpha_{j_1} + \alpha_{j_2}) \cdots (\alpha_{j_1} + \cdots + \alpha_{j_N})} . $$
 Then one defines the following map:
 \begin{equation} \label{deftnDbar}
 \begin{array}{cccc}
  \overline{D} : & \mathbb{C}[N]  & \longrightarrow & \mathbb{C}(\alpha_1, \ldots , \alpha_n) \\
   {}  & f & \longmapsto & \sum_{\mathbf{j}}  \overline{D}_{\mathbf{j}} (f,e_{\mathbf{j}}) . 
  \end{array}
  \end{equation}
  Note that this sum is always finite as $U(\mathfrak{n})$ acts locally nilpotently on $\mathbb{C}[N]$. 
   One of the main results of \cite{BKK} is that the evaluation of $\overline{D}$ on an element $b_{Z}$ of the Mirkovi\'c-Vilonen basis can be related to a certain equivariant multiplicity of the MV cycle $Z$. 
 
   \begin{thm}[{{\cite[Lemma 8.3, Corollary 10.6]{BKK}}}] \label{thmBKK}
    \begin{enumerate}
    \item The map $\overline{D}$ is an algebra morphism. 
    \item  For any $\mu \in  - Q_{+}$ and any stable MV cycle $Z$ of weight  $\mu$, one has 
   $$ \overline{D}(b_{Z}) = \epsilon_{L_{\mu}}^{T^{\vee}}(Z) . $$
       \end{enumerate}
   \end{thm} 
  
   Combining Theorem~\ref{thmBKK} with Brion's Proposition~\ref{propBrion}, we get:
   \begin{cor} \label{Dbarsmooth}
   Let $Z$ be a stable MV cycle of weight $- \nu , \nu \in Q_{+}$. Assume $Z$ is smooth at the point $L_{-\nu}$. Then one has 
  $$ \overline{D}(b_Z) = \frac{1}{ \beta_{i_1} \cdots \beta_{i_r}} $$
  where $\beta_{i_1}, \ldots, \beta_{i_r} \in \Phi_{+}$ are the weights of the action of $T^{\vee}(\mathbb{C})$ on the tangent space of $Z$ at $L_{-\nu}$. 
   \end{cor}
  
   \begin{rk}
 Baumann-Kamnitzer-Knutson prove that the map $\overline{D}$ has the following geometric counterpart ({{\cite[Proposition 8.4]{BKK}}}) . For every regular lement $x$ in $\mathfrak{t}$, the group $N$ acts simply transitively on the subset $x + \mathfrak{n}$ of $\mathfrak{g}$. Hence one can consider the unique  $n_x \in N$ such that $Ad_{n_{x}}(x) = x + e$. Then the algebra morphism  $\overline{D}$ is dual to the morphism of varieties given by 
   $$
    \begin{array}{ccc}
       \mathfrak{t}^{reg} & \longrightarrow & N  \\
                      x  & \longmapsto & n_{x} . 
     \end{array}
     $$
 \end{rk}
 
   \begin{rk}
 The morphism $\overline{D}$ provides a very useful tool to compare various bases of $\mathbb{C}[N]$. For instance Dranowski-Kamnitzer-Morton-Ferguson \cite{BKK} show that the MV basis and the dual semicanonical basis of $\mathbb{C}[N]$ are not the same by exhibiting elements of these bases satisfying some compatibility condition (see {{\cite[Definition 12.1]{BKK}}}) but where $\overline{D}$ nonetheless takes different values. 
 \end{rk}

  \section{Equivariant multiplicities and determinantal modules in $R-gmod$}
  \label{result3}

   This section contains the statements of the main results of this paper. We consider the category $R-gmod$ associated to a simply-laced type Lie algebra $\mathfrak{g}$.  We begin by proving a formula for the equivariant multiplicities of strongly homogeneous modules in this category using Nakada's colored hook formula \cite{Nakada}. We explain why this provides a natural motivation for the study of equivariant multiplicities of the determinantal modules categorifying the flag minors in $\mathbb{C}[N]$.

 \subsection{Equivariant multiplicities and strongly homogeneous modules}    
  \label{prelim}
 
In this subsection we outline a remarkable property of strongly homogeneous modules in the sense of Kleshchev-Ram \cite{KRhom}: the image of their isomorphism classes under the morphism $\overline{D}$ takes a distinguished form, that can be viewed as a generalized version of the Peterson-Proctor hook formula. This property will be useful in Sections~\ref{initAn} and ~\ref{initD4}. 
 
  \bigskip
  
  We consider a  Lie algebra  $\mathfrak{g}$ of finite simply-laced type and we let $W$ denote the corresponding Weyl group. 
  Consider an element of the dual canonical basis of $\mathbb{C}[N]$. It can be written as the isomorphism class of a simple object $M$ in $R-gmod$. Decompose $M$ as the (finite) direct sum of its weight subspaces:
 $$ M = \bigoplus_{\mathbf{j}} e(\mathbf{j}) \cdot M . $$
 Then it follows from the definition of $\overline{D}$ (see~\eqref{deftnDbar}) that its evaluation  on $[M]$ is given by 
 \begin{equation} \label{Dbar}
\overline{D}([M]) = \sum_{\mathbf{j} = (j_1, \ldots ,j_r)} \dim(e(\mathbf{j}) \cdot M) \frac{1}{\alpha_{j_1} (\alpha_{j_1} + \alpha_{j_2}) \cdots (\alpha_{j_1} + \cdots + \alpha_{j_r})} . 
  \end{equation}
In what follows, we will be using this expression to compute the images under $\overline{D}$ of elements of the dual canonical basis. In particular, we can now show that $\overline{D}$ takes remarkable values on the classes of strongly homogeneous modules (see Section~\ref{homogKLR}). 
  
   \begin{prop} \label{propbarD}
   Let $w \in \mathcal{M}in^{+} \subset W$ be a dominant minuscule element of $W$ and let $S(w)$ denote the corresponding strongly homogeneous module in $R-gmod$. Then one has
   $$ \overline{D}([S(w)]) = \prod_{\beta \in \Phi_{+}^{w}} \frac{1}{\beta}  . $$
   \end{prop}
   
    \begin{proof}
    By Theorem~\ref{thmKRhom}, all the weight subspaces of $S(w)$ are one-dimensional and are in bijection with the reduced expressions of $w$. Hence it follows from  Equation~\eqref{Dbar} that 
    $$  \overline{D}([S(w)]) =  \sum_{(j_1, \ldots , j_N) \in \text{Red}(w)} \frac{1}{\alpha_{j_1}} \frac{1}{\alpha_{j_1} + \alpha_{j_2}} \cdots \frac{1}{\alpha_{j_1} + \alpha_{j_2} + \cdots + \alpha_{j_N}} . $$
  This is exactly the right hand side of Nakada's colored hook formula~\eqref{eqNakada} and thus Theorem~\ref{HookNakada} implies
  $$ \overline{D}([S(w)]) = \prod_{\beta \in \Phi_{+}^{w}} \frac{1}{\beta}  . $$
    \end{proof}
   
   \smallskip
   
  The following example shows that this property fails if one considers simple modules (even homogeneous ones) that are not strongly homogeneous.
  
  \begin{ex} \label{examplehomog}
   Consider for instance the simple module $L(231)$ in type $A_3$ (with the natural ordering on good Lyndon words). This module is homogeneous and corresponds to $s_2s_3s_1 \in W$ via the bijection of Theorem~\ref{thmKRhom}. This element is minuscule but not dominant minuscule. It has two reduced expressions, namely $(2,3,1)$ and $(2,1,3)$ and thus one has 
 \begin{align*}
  \overline{D}([S(s_2s_3s_1)]) &= \frac{1}{\alpha_2(\alpha_2 + \alpha_3)(\alpha_1+\alpha_2 + \alpha_3)} + \frac{1}{\alpha_2(\alpha_1 + \alpha_2)(\alpha_1+\alpha_2 + \alpha_3)}  \\
   &= \frac{\alpha_1 + 2\alpha_2 + \alpha_3}{\alpha_2(\alpha_1 + \alpha_2)(\alpha_2 + \alpha_3)(\alpha_1+\alpha_2 + \alpha_3)} . 
   \end{align*}
 \end{ex}

 The main motivations for the present work come from the observations that there seems to be an intimate connection between the set of strongly homogeneous modules $S(w), w \in  \mathcal{M}in^{+}$ and the set determinantal modules categorifying the flag minors of $\mathbb{C}[N]$ (see Section~\ref{reminddeterm}). We will go back to this in the last Section of this paper.
 
 \begin{ex} \label{exampleA3}
    Let us detail the type $A_3$ case as an example. There are $6$ positive roots and the cluster algebra $C[N]$ has 14 seeds. There are $3! = 6$ total orderings on $I = \{1,2,3\}$. We can compute the determinantal modules of the corresponding standard seeds using Theorem~\ref{thmcasbi2}. For instance, the choice of $1<2<3$ yields the following ordering on $\Phi_{+}$:
       $$ \alpha_1 < \alpha_1 + \alpha_2 < \alpha_1  + \alpha_2 +  \alpha_3 <  \alpha_2 <  \alpha_2 +  \alpha_3 < \alpha_3 .$$
       The corresponding reduced expression of $w_0$ is $ {\bf i} = (1,2,3,1,2,1)$ and the determinantal modules of the seed $\mathcal{S}^{\bf i}$ are given by 
  $$ L(1) \qquad L(12) \qquad L(123) \qquad L(21) \qquad L(2312) \qquad L(321) . $$
  Similarly one can repeat this procedure for the other orderings on $I$.  We obtain five distinct standard seeds. It turns out that this is enough to get all the flag minors. We can observe  that the determinantal modules obtained this way are all strongly homogeneous and correspond to the  following Weyl group elements
  $$ s_1, s_2, s_3, s_1s_2, s_2s_3, s_1s_2s_3, s_3s_1s_2, s_3s_2s_1, s_2s_3s_1s_2 . $$
  It is straightforward to check that the homogeneous modules $S(w)$ for $w$ in this list are exactly the prime strongly homogeneous modules in $R-gmod$ (Indeed, the only dominant minuscule element of $\mathfrak{S}_3$ that is missing in this list is $s_3s_1$, but $S(s_3s_1)=L(31)=L(3) \circ L(1)$ is not prime). Thus the determinantal modules categorifying the flag minors coincide with the prime strongly homogeneous modules in this case. 
  \end{ex}
   
   Another interesting case  is where $\mathfrak{g}$ is of type $D_4$, which will be studied in detail in Section~\ref{initD4}. Although the determinantal modules are not necessarily homogeneous in this case, it seems that prime strongly homogeneous modules still categorify certain flag minors.Therefore it is natural to ask the following:
   
    \begin{Question} \label{question}
 Could the distinguished value of $\overline{D}$ be a characteristic of the flag minors, rather than a specific property of classes of strongly homogeneous modules ?  
     \end{Question}
     
      Comparing with the geometric point of view, we know from Corollary~\ref{Dbarsmooth} that the values of $\overline{D}$ take a remarkable form on the  elements of the MV basis corresponding to MV cycles satisfying certain smoothness properties. Thus given a MV cycle $Z$ such that $b_Z$ is a flag minor of $\mathbb{C}[N]$, this also raises the question of the smoothness of $Z$ at $L_{\mu}$ where $\mu \in - Q_{+}$ denotes the weight of $Z$. 
   
   \subsection{The main results}    
      
  Motivated by the discussion of the previous Section, we now state the main results of this paper. They essentially split into two main parts: one consists in exhibiting  several good properties on equivariant multiplicities that propagate under certain cluster mutations, and the other one consists in checking that these properties are indeed satisfied for a particular standard seed in $\mathbb{C}[N]$.     
   Question~\ref{question} suggests the following conjecture:
      
    \begin{conj} \label{weakconj}
       Let $\mathfrak{g}$ be a Lie algebra of finite simply-laced type and let $x$ be a flag minor in $\mathbb{C}[N]$. Then the equivariant multiplicity of $x$ is of the form
       $$ \overline{D}(x) =  \frac{1}{\prod_{\beta \in \Phi_{+}} \beta^{n_{\beta}(x)}}  $$
      where $n_{\beta}(x)$ is a nonnegative integer for every $\beta \in \Phi_{+}$. 
       \end{conj}
       
        The strategy for proving this conjecture is the following. We know that the standard seeds of $\mathbb{C}[N]$ are related to each other by certain cluster mutations, each of which correspond to a change of reduced expression of $w_0$ i.e. there is $k \in \{1, \ldots , N \} $ such that $i_k=p , i_{k+1} = q$ and $i_{k+2}=p$  with $p \cdot q = -1$:
        $$
     \begin{array}{ccc}
         \mathbf{i} = ( \ldots , p , q , p , \ldots )  \quad &  \leadsto & \quad  \mathbf{i}' = ( \ldots , q , p , q , \ldots )  \\
         \s^{\mathbf{i}} = ((x_1, \ldots , x_N) , Q^{\mathbf{i}}) \quad  & \leadsto &  \quad \s^{\mathbf{i}'} = ((x_1, \ldots x_{k-1} , x'_k  , x_{k+1} , \ldots ,  x_N) , Q^{\mathbf{i}'}) . 
             \end{array}
         $$
  The proof is composed of two steps:
 
    \begin{itemize}
     \item  Firstly, we show that the desired result propagates under these cluster mutations. Thus we start by assuming that $\overline{D}(x_j)$ is of the form $1/P_j$ for every $1 \leq j \leq N$.  We want to show that $\overline{D}(x'_k)$ is of the form $1/P'_k$. As $\overline{D}$ is an algebra morphism, one can immediately express $\overline{D}(x'_k)$ in terms of the $P_j$.  But then it is not hard to see that $\overline{D}(x'_k)$ has no reason to take the desired form, at least without any further assumptions on the $P_j$. Therefore, we consider certain relations between the $P_j$ entirely determined by $\mathbf{i}$ implying that $\overline{D}(x'_k)$ is of the form  $1/P'_k$. We prove that these relations are preserved under mutation, i.e. the polynomials $P_1, \ldots , P_{k-1}, P'_k, P_{k+1} , \ldots , P_N$ satisfy the  corresponding relations determined by $\mathbf{i}'$.  Thus this procedure can be iterated to arbitrary sequences of mutations from one standard seed to another. \\
   
    \item  Secondly, we  shall exhibit a particular standard seed in $\mathbb{C}[N]$ whose cluster variables $x_j  , j \in \{1, \ldots , N \}$ have equivariant multiplicities of the form $1/P_j$, and check that the polynomials $P_j$ satisfy the relations required in the previous step. 
    \end{itemize}
    
    The first main result of this paper consists in proving the first step for all simply-laced types. 
    
     \begin{thm} \label{thmpropag}
     Let $\mathfrak{g}$ be any simply-laced type simple Lie algebra and let $\mathbf{i}$ and $\mathbf{i}'$ be two reduced expressions of $w_0$ related as above.
   Assume that the cluster variables  $x_1, \ldots , x_N$ of the  standard seed $\s^{\mathbf{i}}$ satisfy  the following properties:
     \begin{enumerate}[label=(\Alph*)]
   \item For every $1 \leq j \leq N$, the rational fraction $\overline{D}(x_j)$ is of the form $1/P_j$ where $P_j$ is a product of positive roots. 
   \item For every $1 \leq j \leq N$ one has  
   $$ P_j P_{j_{-}(\mathbf{i})} = \beta_j  \prod_{\substack{ l<j<l_{+}(\mathbf{i}) \\ i_l \cdot i_j =-1}} P_l . $$  
    with $P_0 := 1$. 
   \item For every $j \in J_{ex}$ and every $1 \leq i \leq N$, one has $ (\beta_i ; P_j) - (\beta_i ; P_{j_{+}(\mathbf{i})}) \leq 1 $. 
  \end{enumerate}
  Then the cluster variables of $\s^{\mathbf{i}'}$ satisfy the analogous properties determined by $\mathbf{i}'$. 
  \end{thm}

   \begin{rk}
Note that we allow the left hand-side of the inequality (C) to be negative. Moreover, this inequality is relevant only in types $D_n$ and $E_6, E_7, E_8$ as the $P_j$ are always multiplicity-free in type $A_n$ (for any standard seed). 
\end{rk}
  
   \begin{rk}  \label{rkBalgo}
   Let us outline a curious consequence of Property (B). The polynomial $P_{j_{-}}$ as well as all the polynomials involved in the right hand side are all of the form $P_l$ for some $l$ strictly smaller than $j$. Thus if the Property (B) holds for the standard seed $\s^{\mathbf{i}}$, then for every $j$, $P_j$ is entirely determined by $\beta_j$ and by the $P_l , l<j$. By a straightforward induction, this implies that the polynomials $P_1 , \ldots ,P_N$ are in fact entirely determined by the data of the $P_j$ such that $j_{-}=0$, i.e. the positions of the first occurrence in $\mathbf{i}$ of each letter of $I$. Interestingly, these polynomials are in general easy to compute by hand, as illustrated in the next example. For instance it follows from {{\cite[Proposition 3.14]{KK}}} that the corresponding determinantal modules are cuspidal for the convex ordering on $\Phi_{+}$ associated to $\mathbf{i}$.  In particular when $\mathbf{i}$ comes from a total ordering on $I$, Theorem~\ref{thmcasbi2} explicitly provides the good Lyndon words associated to these modules. In simply-laced types such cuspidal modules are often well understood  (see for example {{\cite[Section 8]{KR}}}). Hence  the value of $\overline{D}$ on their isomorphism class can be computed using Equation~\eqref{Dbar}. Thus Property (B) provides an algorithm allowing to compute the equivariant multiplicities of all the flag minors belonging to one standard seed, which was a non-trivial computation a priori. 
    \end{rk}
  
 \begin{ex}
 Let us provide an example of a direct computation of $P_j$ for $j$ such that $j_{-}(\mathbf{i}) = 0$. 
 
  Consider $\mathfrak{g}$ of type $D_4$ and take $\mathbf{i} = \mathbf{i}_{nat} = (1,3,2,4,3,1,4,3,2,4,3,4)$ as in Section~\ref{initD4} below. Consider for instance the first occurrence of the letter $3$, i.e. $j=2$. The determinantal module $M_2^{\mathbf{i}_{nat}}$ is the cuspidal module $L(13)$ (with respect to the natural ordering). This can be deduced from Theorem~\ref{thmcasbi2}, or can be checked by hand using for instance {{\cite[Proposition 10.2.4]{KKKO}}}. This module is one-dimensional as a vector space (see for instance {{\cite[Section 8.7]{KR}}}) and thus its weight subspace decomposition is simply $ L(13) = \mathbb{C} \cdot \mathbf{v}_{13}$ with $e(13) \cdot \mathbf{v}_{13} = \mathbf{v}_{13}$. Thus it is immediate to compute its image under $\overline{D}$ using Equation~\ref{Dbar} and we get $P_2 = \alpha_1(\alpha_1 + \alpha_3)$. 
  \end{ex}
 
  The second main result of this paper is to check the second step for types $A_n$  and $D_4$. 
  
   \begin{thm}  \label{thminitcond}
 Assume $\mathfrak{g}$ is of type $A_n , n \geq 1$ or $D_4$. Let $\mathbf{i}_{nat}$ denote the reduced expression of $w_0$ corresponding to the natural ordering on the vertices of the Dynkin diagram of $\mathfrak{g}$. Then Properties (A), (B) and (C) hold for the standard seed $\s^{\mathbf{i}_{nat}}$. 
 \end{thm} 

In particular, this proves Conjecture~\ref{weakconj} in types $A_n$ and $D_4$. 
Section~\ref{propag} is devoted to the proof of Theorem~\ref{thmpropag} and Sections~\ref{initAn} and~\ref{initD4} respectively contain the proofs of Theorem~\ref{thminitcond} in types $A_n$ and $D_4$. 

 As we saw in Section~\ref{prelim}, Property (A) fails for homogeneous modules that are not strongly homogeneous. Further computations in types $A_n$ and $D_4$ tend to suggest that the evaluation of  $\overline{D}$ on cluster variables of $\mathbb{C}[N]$  fails to take the remarkable form of (A) for the cluster variables that are not flag minors (i.e. that do not appear in any of the standard seeds $\s^{\mathbf{i}} , \mathbf{i} \in \text{Red}(w_0)$. Therefore, we propose the following stronger version of Conjecture~\ref{weakconj}. 
  
  \begin{conj} \label{strongconj}
   Let $\mathfrak{g}$ be a Lie algebra of finite simply-laced type and let $w_0$ be the longest element of the associated Weyl group. Then the flag minors are exactly the cluster variables of $\mathbb{C}[N]$ whose image under $\overline{D}$ are of the form  
   $$  \frac{1}{\prod_{\beta \in \Phi_{+}} \beta^{n_{\beta}}} . $$
   for some family of nonnegative integers $(n_{\beta})_{\beta \in \Phi_{+}}$.
 \end{conj}

 In other words, flag minors would essentially be characterized by the distinguished form of their equivariant multiplicities. 
 Let us illustrate this by an example. 
 
  \begin{ex}
  Consider $\mathfrak{g}= \mathfrak{sl}_4$. As mentioned in Example~\ref{exampleA3}, the cluster structure of $\mathbb{C}[N]$ contains $14$ seeds. The  simple objects of $R-gmod$ corresponding to the cluster variables are given as follows (with respect to the natural ordering $1<2<3$):
$$ L(1), L(2), L(3), L(12), L(21), L(23), L(32), L(312), L(231), L(123), L(321), L(2312) . $$
 The three last modules correspond to the frozen variables. Among the unfrozen cluster variables, the only one that is not a flag minor (i.e. that does not appear in any standard seed) is $[L(231)]$, and as we saw in Example~\ref{examplehomog} the rational fraction $\overline{D}([L(231)])$ is not of the form (A).
 \end{ex}
    
    \section{Propagation under cluster mutation}
  \label{propag}
    
 This section is devoted to the proof of Theorem~\ref{thmpropag}. 
Let $\mathfrak{g}$ be a simple Lie algebra of finite simply-laced type. We choose a labeling $I$ of the vertices of the Dynkin diagram of $\mathfrak{g}$ and we write the associated Cartan datum as 
$$ i \cdot i = 2 \quad \text{and} \quad i \cdot j = -1 \Leftrightarrow \text{$i$ and $j$ are neighbours in the Dynkin diagram of $\mathfrak{g}$} . $$
 We consider the longest element $w_0$ of the corresponding Weyl group $W$ and we denote by $N$ the length of $w_0$. We fix a reduced expression $\mathbf{i} = (i_1, \ldots , i_N)$ of $w_0$. Let $x_1, \ldots , x_N$ denote the cluster variables and $Q^{\mathbf{i}}$ denote the quiver of the standard seed $\s^{\mathbf{i}}$ of $\mathbb{C}[N]$. We also let $M_1, \ldots , M_N$ denote the determinantal modules corresponding to $\mathbf{i}$ i.e. the simple modules in $R-gmod$ whose isomorphism classes are $x_1, \ldots , x_N$ in $K_0(R-gmod) \simeq \mathbb{C}[N]$. 
   
   The quiver $Q^{\mathbf{i}}$ is defined as follows (see \cite{GLS,KKKO}). First  recall the following piece of notations: for any $1 \leq j \leq N$ we set 
   $$ j_{-}(\mathbf{i}) := \max \left( \{ l, 1 \leq l<j , i_l=i_j \} \sqcup \{ 0 \} \right) $$
   and
$$  j_{+}(\mathbf{i}) := \min \left( \{l, N \geq l>j , i_l=i_j \} \sqcup \{ N+1 \} \right) . $$
   We outline the dependence on $\mathbf{i}$ in order to avoid confusion later in the proof, as we will be considering two different reduced expressions $\mathbf{i}$ and $\mathbf{i}'$. 
    The index set of $Q^{\mathbf{i}}$ is $J = \{1, \ldots , N \}$, which splits into a frozen part $J_{fr}(\mathbf{i}) := \{ u \in J \mid u_{+}(\mathbf{i}) = N+1 \}$ and an unfrozen part  $J_{ex}(\mathbf{i}) := J \setminus J_{fr}(\mathbf{i})$.
 The set of arrows of $Q^{\mathbf{i}}$ is composed of two different kinds of arrows: 
    \begin{itemize}
     \item the \textit{ordinary arrows}: there is such an arrow from the vertex $u$ to the vertex $v$ if and only if 
     $$ i_u \cdot i_v = -1 \quad \text{and} \quad u < v < u_{+}(\mathbf{i}) < v_{+}(\mathbf{i}) . $$
     \item the \textit{horizontal arrows}: for every $u \in J_{ex}(\mathbf{i})$ there is an arrow from the vertex $u_{+}(\mathbf{i})$ to the vertex $u$. 
    \end{itemize}
 For every $j \in J$ we let $\text{in}(j)$ (resp. $\text{out}(j)$) denote the set of all indices $l$ such that there is an arrow from $l$ to $j$ (resp. from $j$ to $l$) in $Q^{\mathbf{i}}$ and  $\text{inord}(j)$ (resp. $\text{outord}(j)$) the set of all indices $l$ such that there is an ordinary arrow from $l$ to $j$ (resp. from $j$ to $l$) in $Q^{\mathbf{i}}$. We have 
 $$ \text{in}(j) = \text{inord}(j) \sqcup \{ j_{+}(\mathbf{i}) \} \quad \text{and} \quad \text{out}(j) = \text{outord}(j) \sqcup \{ j_{-}(\mathbf{i}) \} . $$
  We now consider another reduced expression $\mathbf{i}'$ of $w_0$ such that $\mathbf{i}'$ is obtained from $\mathbf{i}$ by performing a braid relation in $W$ i.e. there is $k \in \{1, \ldots , N \}$ such that $i_k=p, i_{k+1}=q, i_{k+2}=p$ with $p,q \in I$ such that $p \cdot q = -1$. 
  $$ \mathbf{i} = (i_1, \ldots , i_{k-1} , p , q , p , \ldots ) \quad \leadsto \quad  \mathbf{i}' = (i_1, \ldots , i_{k-1} , q , p , q , \ldots ) . $$
  We denote by $\beta'_1, \ldots , \beta'_N$ the positive roots given by the reduced expression $\mathbf{i}'$ and $x'_1, \ldots , x'_N$ the cluster variables of the  standard seed $\s^{\mathbf{i}'}$.
The seeds $\s^{\mathbf{i}}$ and $\s^{\mathbf{i}'}$ are related by a one-step mutation in the direction $k$. It is appropriate for our purpose to be slightly careful with the labeling of the vertices of $Q^{\mathbf{i}}$ and $Q^{\mathbf{i}'}$.

We let $\bs$ denote the transposition $(k+1,k+2)$ of $\{1, \ldots , N \}$ i.e. the permutation that exchanges the indices $k+1$ and $k+2$ and leaves the others fixed. 
The set of vertices of the quiver $Q^{\mathbf{i}'}$ is $\bs(J)$. In other words the vertex labeled $k+1$ in $Q^{\mathbf{i}'}$ is in fact the vertex labeled $k+2$ in $Q^{\mathbf{i}}$ and vice versa. 
 There is an arrow $i \rightarrow j$ in $Q^{\mathbf{i}'}$ if and only if there is an arrow $\bs(i) \rightarrow \bs(j)$ in the quiver obtained from $Q^{\mathbf{i}}$ by the usual mutation process. 
$$
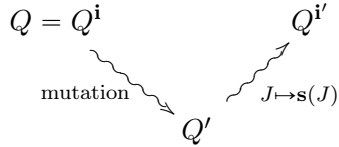
\xymatrix{ Q = Q^{\mathbf{i}} \ar@{~>}[rd]_{\text{mutation}} & {} & Q^{\mathbf{i}'} \\
       {} & Q' \ar@{~>}[ru]_{J \mapsto \bs(J)} & {}   }
 $$
 We begin with a couple of elementary Lemmas as prerequisites for the proof. 
 
 \begin{lem} \label{mutbetaM}
    The positive roots $\beta'_1, \ldots , \beta'_N$ are related to the $\beta_j$ as follows:
   $$ \beta'_k = \beta_{k+2} \quad \beta'_{k+1} = \beta_{k+1} \quad \beta'_{k+2} = \beta_k \quad \text{$\beta'_j= \beta_j$ for any $j \notin \{k,k+1,k+2\}$} . $$
    The flag minors $M'_1 , \ldots , M'_N$ are given by:
    $$ x'_{k+1} = x_{k+2}  \quad x'_{k+2} = x_{k+1} \quad \text{$x'_j = x_j$ for any  $j \notin \{k,k+1,k+2\}$} . $$
    \end{lem}
    
     \begin{proof}
   Recall that $\beta_j = s_{i_1} \cdots s_{i_{j-1}} \alpha_{i_j}$ and similarly for the the $\beta'_j$ with $\mathbf{i}'$. Thus 
   $$ \beta'_k = s_{i'_1} \cdots s_{i'_{k-1}} \alpha_{i'_k} = s_{i_1} \cdots s_{i_{k-1}} \alpha_q = s_{i_1} \cdots s_{i_{k-1}}  s_p s_q \alpha_p = \beta_{k+2} . $$
    One checks the other cases in a similar way. 
   
   The flag minor $M_j$ can be written as $D \left( s_{i_1} \cdots s_{i_j} \omega_{i_j} , \omega_{i_j} \right)$ with the notations of Section~\ref{reminddeterm}. Thus one has for instance 
\begin{align*}
 x'_{k+1} &= D \left( s_{i'_1} \cdots s_{i'_{k+1}} \omega_{i'_{k+1}} , \omega_{i'_{k+1}} \right) =  D \left( s_{i_1} \cdots s_{i_{k-1}} s_q s_p \omega_p , \omega_p \right) \\
  &= D \left( s_{i_1} \cdots s_{i_{k-1}} s_p s_q s_p \omega_p , \omega_p \right)  = x_{k+2} . 
 \end{align*}
  The other cases can be checked in the same way. 
     \end{proof}
    
     \smallskip

 \begin{lem} \label{beta}
 One has 
$$ \beta_k + \beta_{k+2} = \beta_{k+1} . $$
 \end{lem}
 
  \begin{proof}
  This is a straightforward consequence of the definition of the $\beta_j$. Indeed,
  $$ \beta_k + \beta_{k+2} = s_{i_1} \cdots s_{i_{k-1}} \left( \alpha_p + s_p s_q(\alpha_p) \right) =  s_{i_1} \cdots s_{i_{k-1}} \left( \alpha_p + \alpha_q \right) =  s_{i_1} \cdots s_{i_{k-1}} s_p (\alpha_q) $$
  i.e.  $\beta_k + \beta_{k+2}  = \beta_{k+1} . $
  \end{proof}
 
   \smallskip
 
  The structure of the proof is summarized in Figure~\ref{schema}. Let us briefly explain the main points. The starting assumption is that for each $j \in J$, the rational function $\overline{D}(x_j)$ is of the form $1/P_j$ where $P_j$ is a product of positive roots. The aim is to prove that  $\overline{D}(x'_k)$ is also of the form $1/P'_k$ where $P'_k$ is a product of positive roots. It is not hard to convince oneself that this has no chance to hold without any further assumption on the $P_j, 1 \leq j \leq N$. Thus the idea is that the $P_j$ shall satisfy certain relations entirely determined by $\mathbf{i}$ implying the desired form for $\overline{D}(x'_k)$.  Moreover these relations must be preserved under mutation i.e. the polynomials $P_1, \ldots , P_{k-1}, P'_k, P_{k+1}, \ldots , P_N$ have to satisfy the analogous relations determined by $\mathbf{i}'$.
 Therefore we assume that the following properties are satisfied:
  \begin{enumerate}[label=(\Alph*)]
   \item For every $1 \leq j \leq N$, the rational fraction $\overline{D}(x_j)$ is of the form $1/P_j$ where $P_j$ is a product of positive roots. 
   \item For every $1 \leq j \leq N$ one has 
   $$ P_j P_{j_{-}(\mathbf{i})} = \beta_j  \prod_{\substack{ l<j<l_{+}(\mathbf{i}) \\ i_l \cdot i_j =-1}} P_l  . $$
   \item For every $j \in J_{ex}(\mathbf{i})$, one has 
   $$ (\beta_i ; P_j) - (\beta_i ; P_{j_{+}(\mathbf{i})}) \leq 1 $$
    for every $1 \leq i \leq N$, where we set $P_0 := 1$. 
  \end{enumerate}
  We prove that the flag minors of the seed $\s^{\mathbf{i}'}$ satisfy the analogous properties determined by $\mathbf{i}'$ and denoted (A'), (B') and (C'), as shown in Figure~\ref{schema}.
  
   \smallskip
   
     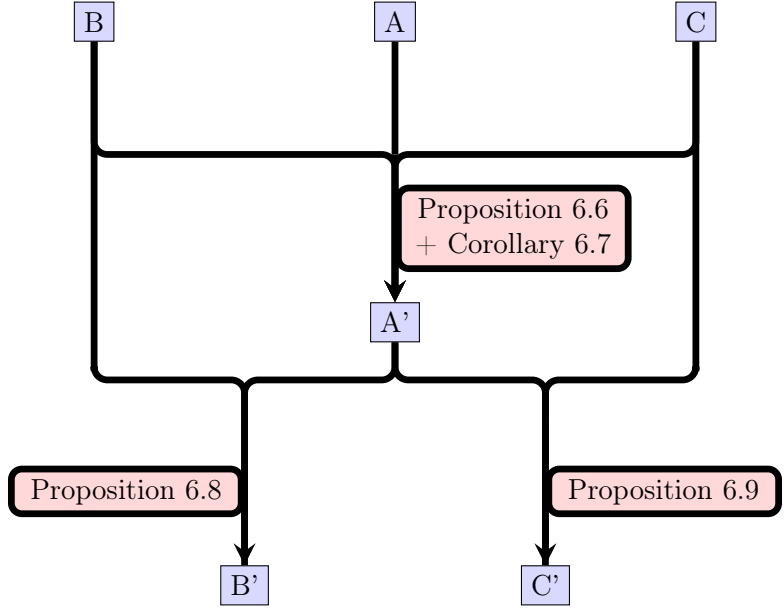
\begin{figure}
  
  \center
  
\begin{tikzpicture}
\tikzstyle{letter}=[above,rectangle,align=center,draw,fill=blue!15]
\tikzstyle{arrow}=[->,>=stealth,rounded corners=5pt,very thick,line width=0.9mm]
\tikzstyle{prop}=[above,rectangle,align=center,draw,text width=2.8cm,fill=red!15,rounded corners]

\node[letter] (A) at (0,4) {A};
\node[letter] (B) at (-4,4) {B};
\node[letter] (C) at (4,4) {C};
\node[letter] (Aprime) at (0,0) {A'};
\node[letter] (Bprime) at (-2,-3.5) {B'};
\node[letter] (Cprime) at (2,-3.5) {C'};

\draw[line width=0.9mm] (A)--(0,2.5);
\draw[arrow] (0,2.5) to node[prop,right]{Proposition~\ref{division} + Corollary~\ref{corPprimek}} (Aprime);
\draw[arrow] (B) |- (-2,2.5) -| (Aprime);
\draw[arrow] (C) |- (2,2.5) -| (Aprime);
\draw[arrow] (B) |- (-4,-0.5) -| (Bprime);
\draw[arrow] (C) |- (4,-0.5) -| (Cprime);
\draw[rounded corners=5pt,very thick,line width=0.9mm] (Aprime) |- (-1,-0.5) -| (-2,-1);
\draw[rounded corners=5pt,very thick,line width=0.9mm] (Aprime) |- (1,-0.5) -| (2,-1);
\draw[very thick,line width=0.9mm] (-2,-1) to node[prop,left]{Proposition~\ref{secondprop}}  (Bprime);
\draw[very thick,line width=0.9mm] (2,-1) to node[prop,right]{Proposition~\ref{mutmult}} (Cprime);
\end{tikzpicture}

 \caption{Structure of the proof of Theorem~\ref{thmpropag}.}
 
  \label{schema}

\end{figure}

  In what follows we will be using the following notations: for every $j \in J$ we set
  $$ \Pij := \prod_{l \in \text{in}(j)} P_l \qquad   P_{\text{inord}(j)} := \prod_{l \in \text{inord}(j)} P_l $$
  $$  \Poj := \prod_{l \in \text{out}(j)} P_l \qquad P_{\text{outord}(j)} := \prod_{l \in \text{outord}(j)} P_l . $$
  Let us begin with a straightforward consequence of Property (B) that will be useful throughout the proof. 

\begin{lem} \label{Dbarhat}
For every $j \in J_{ex}(\mathbf{i})$, one has
$$ \beta_j \Pij = \beta_{j_{+}(\mathbf{i})} \Poj  . $$
\end{lem}

\begin{proof}
We fix $j \in J_{ex}(\mathbf{i})$. In particular $j_{+}(\mathbf{i}) \leq N$ so we can combine Property (B) at ranks $j$ and $j_{+}(\mathbf{i})$:
 $$ P_j P_{j_{-}(\mathbf{i})} = \beta_j  \prod_{\substack{ l<j<l_{+}(\mathbf{i}) \\ i_l \cdot i_j =-1}} P_l  \quad \text{and} \quad P_{j_{+}(\mathbf{i})} P_j = \beta_{j_{+}(\mathbf{i})}  \prod_{\substack{ l<j_{+}(\mathbf{i})<l_{+}(\mathbf{i}) \\ i_l \cdot i_j =-1}} P_l  . $$
 Dividing the first one by the second one we get
 $$ \frac{P_{j_{-}(\mathbf{i})}}{P_{j_{+}(\mathbf{i})}} =   \beta_j \prod_{\substack{ l<j<l_{+}(\mathbf{i})<j_{+}(\mathbf{i}) \\ i_l \cdot i_j =-1}} P_l   \left( \beta_{j_{+}(\mathbf{i})} \prod_{\substack{ j<l<j_{+}(\mathbf{i})<l_{+}(\mathbf{i}) \\  i_l \cdot i_j =-1}} P_l \right)^{-1} = \frac{\beta_j P_{\text{inord}(j)}}{\beta_{j_{+}(\mathbf{i})} P_{\text{outord}(j)}} . $$
 By definition one has $\text{in}(j) = \{ j_{+}(\mathbf{i}) \} \sqcup \text{inord}(j)$ and $\text{out}(j) = \{ j_{-}(\mathbf{i}) \} \sqcup \text{outord}(j)$. Thus we have proven the desired statement. 
\end{proof}

\smallskip

\begin{rk} \label{rkBdom}
This statement can be rephrased in a cluster-theoretic way as follows. Recall Fomin-Zelevinsky's notation $\yjh$. They are Laurent monomials in the $x_i$ that can be written in our case as 
$$ \yjh := \prod_{i \in \text{in}(j)} x_i  \prod_{i \in \text{out}(j)} x_i^{-1} $$
 for every $j \in J_{ex}(\mathbf{i})$. The algebra morphism $\overline{D}$ can be extended to the fraction field of $\mathbb{C}[N]$ so that the evaluation $\overline{D}(\yjh)$ makes sense. Lemma~\ref{Dbarhat} can then be restated as 
$$ \overline{D}(\yjh) = \frac{\beta_j}{\beta_{j_{+}(\mathbf{i})}} $$
for every $j \in J_{ex}(\mathbf{i})$. 
 \end{rk}
 
  \bigskip
 
  We let $\Pjt$ denote the polynomial given by Lemma~\ref{Dbarhat} i.e.
   $$ \Pjt :=  \beta_j \Pij = \beta_{j_{+}(\mathbf{i})} \Poj $$
    for every $j \in J_{ex}(\mathbf{i})$. 
 In the sequel of the proof, we denote by $(\beta ; P)$ the multiplicity of $\beta$ in $P$ for every $\beta \in \Phi_{+}$ and $P \in \{P_1, \ldots , P_N\}$, i.e. the largest positive integer $d$ such that $\beta^d \mid P$.
 Let us state another useful consequence of Property (B). The proof is straightforward by induction on $j$, and does not make use of Property (C). 
 
 \begin{lem} \label{cormult}
 For every $i,j \in J_{ex}(\mathbf{i})$, one has
$$ i>j \Rightarrow (\beta_i ; P_j) = 0   \qquad \text{and} \qquad (\beta_i ; P_i) = 1 . $$
 \end{lem}
 
 \begin{proof}
 The proof relies on an induction on $j$. Consider the case $j=1$ and let $h$ denote the letter $h=i_1$.  First note that $\beta_1 = \alpha_h$. The determinantal module $M_1$ is the cuspidal module $L(h)$: it is one-dimensional as a $\mathbb{C}$-vector space and its weight-space decomposition is simply 
 $$ M_1 = \mathbb{C} \mathbf{v}_h \quad \text{with $e(h) \cdot \mathbf{v}_h = \mathbf{v}_h$} . $$
  Thus Equation~\eqref{Dbar} yields
 $$ \overline{D}(M_1) = \frac{1}{\alpha_h} = \frac{1}{\beta_1} . $$
 Equivalently $P_1 = \beta_1$ and the desired statement holds for $j=1$.
 
  Fix $j \in  \{1, \ldots , N \}$ and assume the statement is true for every $j' < j$. If $i>j$ then the induction hypothesis implies 
  $$ (\beta_i;P_{j_{-}(\mathbf{i})}) = 0 \quad \text{and} \quad \text{$(\beta_i;P_l) = 0$ for every $l<j<l_{+}(\mathbf{i})$} . $$
  Hence Property (B) implies $(\beta_i;P_j)=0$. 
  If $i=j$ then we use the induction hypothesis in the same way but the presence of $\beta_j$ in the right hand side of Property (B) yields $(\beta_j ; P_j) = 1$. 
 \end{proof}
 
 \smallskip
 
 The following statement is the heart of the proof. The idea is to combine Property (B) together with Lemma~\ref{Dbarhat} at different indices and then to use Property (C). This last property plays a crucial role and is a priori not redundant with Properties (A) and (B). 
 
  \begin{prop} \label{division}
Let $l \in \{1, \ldots , N\}$ and assume there exists $m \in \{1, \ldots , N\}$ such that
 $$ i_l \cdot i_m = -1 \qquad \text{and} \qquad m_{-}(\mathbf{i})<l<m<l_{+}(\mathbf{i})<m_{+}(\mathbf{i}) . $$
  Then $P_l \mid \Plt$. 
 \end{prop}
 
 The assumptions relating $l$ and $m$ mean that there exists an ordinary arrow from the vertex $l$ to the vertex $m$, as well as from the vertex $m_{-}(\mathbf{i})$ to the vertex $l$ (if $m_{-}(\mathbf{i})>0$) in the quiver $Q^{\mathbf{i}}$.
 
  \begin{proof}
  Proving $P_l \mid \Plt$ is equivalent to proving that  $(\beta_i ; P_l) \leq (\beta_i ; \Plt)$ for every $1 \leq i \leq N$. If $i>l$ then $(\beta_i ; P_l) = 0$ by Lemma~\ref{cormult} and thus there is nothing to prove in this case. Assume $i \leq l$.
  
   We deal with the case $i=l$ separately. By Lemma~\ref{cormult}, one has $(\beta_l ; P_l) = 1$. Hence we have 
   $$ (\beta_l ; P_l) = 1 = (\beta_l ; \beta_l) \leq (\beta_l ; \Plt) $$ 
  which is the desired inequality.
  
 From now on we assume $i<l$. If $(\beta_i ; P_l) \leq (\beta_i ; P_{l_{+}(\mathbf{i})})$ then $(\beta_i ; P_l) \leq (\beta_i ; \Plt)$ as $P_{l_{+}(\mathbf{i})} \mid \Plt$ by definition of $\Plt$. Hence we can assume $(\beta_i ; P_l) > (\beta_i ; P_{l_{+}(\mathbf{i})})$. 
  
   Assume $(\beta_i ; P_{\text{inord}(l)}) = 0$. Then applying Lemma~\ref{Dbarhat} at $j=l$ we can write 
  $$ (\beta_i ; P_{l_{+}(\mathbf{i})}) = (\beta_i ; P_{l_{-}(\mathbf{i})}) + (\beta_i ; P_{\text{outord}(l)}) 
    = (\beta_i ; P_{l_{-}(\mathbf{i})}) + (\beta_i ; P_m) + (\beta_i ; R) $$
 where $R$ is the product of the polynomials attached to the other tails of ordinary arrows coming out of $l$ in $Q^{\mathbf{i}}$. Now we use Property (B) at $j=m$: as $i<l$, in particular $i<m$ and hence we get
$$ (\beta_i ; P_{m}) = -(\beta_i ; P_{m_{-}(\mathbf{i})}) + \sum_{\substack{h<m<h_{+}(\mathbf{i}) \\ i_h \cdot i_m = -1}} (\beta_i ; P_h). $$
 By assumption we have $i_l \cdot i_m = -1$ and $l<m<l_{+}(\mathbf{i})$. Hence the previous equation can be written as 
 $$ (\beta_i ; P_{m}) = -(\beta_i ; P_{m_{-}(\mathbf{i})}) + (\beta_i ; P_l) + (\beta_i ; Q) $$ 
 where $Q$ is the product of the $P_h , h \neq l , i_h \cdot i_m = -1 , h<m<h_{+}(\mathbf{i})$. Thus we get
 $$ (\beta_i ; P_{l_{+}(\mathbf{i})}) = (\beta_i ; P_{l_{-}(\mathbf{i})}) -(\beta_i ; P_{m_{-}(\mathbf{i})}) + (\beta_i ; P_l) + (\beta_i ; Q) + (\beta_i ; R) $$
 which can be rewritten as 
  $$ (\beta_i ; P_{m_{-}(\mathbf{i})})  = (\beta_i ; P_{l_{-}(\mathbf{i})})  + (\beta_i ; P_l) - (\beta_i ; P_{l_{+}(\mathbf{i})})   + (\beta_i ; Q) + (\beta_i ; R) . $$
  Since we assumed $(\beta_i ; P_l) > (\beta_i ; P_{l_{+}(\mathbf{i})})$, this implies in particular 
$$ (\beta_i ; P_{m_{-}(\mathbf{i})})  > (\beta_i ; P_{l_{-}(\mathbf{i})})  + (\beta_i ; Q) + (\beta_i ; R) \geq 0 . $$
If $m_{-}(\mathbf{i})=0$ then $P_{m_{-}(\mathbf{i})}=1$ so this is a contradiction. If $m_{-}(\mathbf{i})>0$, then there is an ordinary arrow from the vertex $m_{-}(\mathbf{i})$ to the vertex $l$, and hence one has in particular $(\beta_i ; P_{\text{inord}(l)}) > 0$ which is again a contradiction. 

 Thus we have proven that $(\beta_i ; P_{\text{inord}(l)}) > 0$. This is where Property (C) is crucially involved. Indeed, Property (C) with the label $l$ yields $(\beta_i ; P_l) - (\beta_i ; P_{l_{+}(\mathbf{i})}) \leq 1$. Therefore, we have 
 $$ (\beta_i ; P_{\text{inord}(l)}) \geq 1 \geq (\beta_i ; P_l) - (\beta_i ; P_{l_{+}(\mathbf{i})}) . $$
As $\text{inord}(l) \sqcup \{l_{+}\} = \text{in}(l)$, this implies $(\beta_i ; P_l) \leq (\beta_i ; P_{\text{in}(l)})$ which finishes the proof. 
  \end{proof}
  
  \smallskip
  
   As a straightforward application, we can now prove that Property (A) holds for the seed $\s^{\mathbf{i}'}$. 
 
  \begin{cor} \label{corPprimek}
   One has 
  \begin{equation} \label{Aprime} \tag{A'}
   \overline{D}(x'_k) = \frac{1}{P'_k} \quad \text{with $P'_k$ product of positive roots given by $P'_k = \frac{\Pkt}{\beta_{k+1}P_k}$.} 
     \end{equation}
  \end{cor}
  
   \begin{proof}
   The exchange relation at the vertex $k$ can be written as 
   $$ x_k x'_k = \prod_{j \in \text{in}(k)} x_j + \prod_{j \in \text{out}(k)} x_j . $$
   Applying the algebra morphism $\overline{D}$ we get 
  \begin{align*}
   \frac{\overline{D}(x'_k)}{P_k} &= \frac{1}{\Pik} \left( 1 + \frac{\Pik}{\Pok} \right) 
   = \frac{1}{\Pik} \left( 1 +  \frac{\beta_{k+2}}{\beta_k}  \right)  \text{by Lemma~\ref{Dbarhat} } \\
   &= \frac{1}{\Pik} \frac{\beta_{k+1}}{\beta_k} \text{by Lemma~\ref{beta}}  . 
   \end{align*}
   Thus we get 
   $$ \overline{D}(x'_k) = \frac{\beta_{k+1}P_k}{\Pkt} . $$
Recall from the beginning of this section that we have $i_k=p=i_{k+2}$ and $i_{k+1}=q$ with $p \cdot q = -1$. Hence we can apply Proposition~\ref{division} with $l=k$ and $m=k+1$, which yields $P_k \mid \Pkt$. Lemma~\ref{cormult} implies $(\beta_{k+1} ; P_k) = 0$ and $(\beta_{k+1} ; P_{k+1}) = 1$. Applying Lemma~\ref{Dbarhat} we get 
   $$ (\beta_{k+1} ; P_\text{in}(k)) = (\beta_{k+1} ; P_\text{out}(k)) \geq (\beta_{k+1} ; P_{k+1}) > 0 . $$
   Hence $\beta_{k+1} \mid P_{\text{in}(k)}$. Finally we have $\beta_{k+1} P_k \mid \Pkt$ which proves the Corollary. 
   \end{proof}
 
  For $j \neq k$, we have $x'_j=x_{\bs(j)}$ and hence Property (A) holds with $P'_j := P_{\bs(j)}$.
 Now we prove that Property (B) propagates under mutation.

 \begin{prop} \label{secondprop}
  Property (B) holds for the seed $\s^{\mathbf{i}'}$ i.e. one has
\begin{equation}  \label{secondhyp} \tag{B'}
P'_jP'_{j_{-}(\mathbf{i}')} = \beta'_j  \prod_{\substack{ l<j<l_{+}(\mathbf{i}') \\ i'_l \cdot i'_j =-1}} P'_l
  \end{equation}
 for every $j \in J$. 
 \end{prop}
 
 \begin{proof}
 Throughout this proof, we set for every $j \in J$:
 $$ I(j) := \{ l \in J \mid  i_l \cdot i_j =-1 , l<j<l_{+}(\mathbf{i}) \} ,  \qquad I'(j) := \{ l \in J \mid  i'_l \cdot i'_j =-1 , l<j<l_{+}(\mathbf{i}') \} . $$
 First note that it is sufficient to prove only the desired statement when $j \in \{k, k+1, k+2\}$. Indeed, for other values of $j$, we have $P'_j=P_j$ and $\beta'_j=\beta_j$ by Lemma~\ref{mutbetaM}. As $j \neq k+2$ we have $j_{-}(\mathbf{i}') \neq k$ and hence Lemma~\ref{mutbetaM} implies that $P'_{j_{-}(\mathbf{i}')}=P_{j_{-}(\mathbf{i})}$. Similarly, as $j \neq k+1$, one has $k \notin I'(j)$ and hence $\prod_{l \in I'(j)} P'_l=\prod_{l \in I(j)} P_l$. So if $j \notin \{k, k+1, k+2\}$ then the relations $(B)$ at the index $j$ are exactly the same for $\mathbf{i}$ and $\mathbf{i}'$.

 Consider the case $j=k$. We let $r$ (resp. $s$) denote the position of the last occurrence of the letter $p$ (resp. $q$) strictly before the position $k$. In other words $r=k_{-}(\mathbf{i})=(k+1)_{-}(\mathbf{i}')$ and $s=(k+1)_{-}(\mathbf{i})=k_{-}(\mathbf{i}')$. It is straightforward to check that 
 $$ I(k+1) \sqcup \left( J \cap \{ r \} \right)  = I'(k) \sqcup \{ k \} $$
 (recall that $J := \{1, \ldots, N \}$).
 Indeed if $l$ is such that $i_l \neq p$ then $l \in I'(k)$ if and only if $l \in I(k+1)$. If $i_l=p$, then $l \in I'(k)$ if and only if $r \neq 0$ and $l=r$; on the other hand $l \in I(k+1)$ if and only if $l=k$.
 
  As all the indices in $I'(k)$ are strictly smaller than $k$, one has $P'_l=P_l$ for every $l \in I'(k)$. Thus we have
$$ \prod_{l \in I(k+1)} P_l = \frac{P_k}{P_r} \prod_{l \in I'(k)} P_l =  \frac{P_k}{P_r} \prod_{l \in I'(k)} P'_l  $$
(recall that $P_r=1$ if $r=0$). 
Now we can write 
 $$ P'_kP'_{k_{-}(\mathbf{i}')} = P'_k P_{s} = P_s \frac{\beta_{k+2}P_{k+1}P_{r}}{\beta_{k+1}P_k} \quad \text{by Equation~\eqref{Aprime}}  . $$
 Using Property (B) at the index $k+1$ we get
 $$ P_{k+1} P_s = \beta_{k+1} \prod_{l \in I(k+1)} P_l $$
 and thus 
$$ P'_kP'_{k_{-}(\mathbf{i}')} = \beta_{k+2} \frac{P_r}{P_k}  \prod_{l \in I(k+1)} P_l  = \beta_{k+2} \prod_{l \in I'(k)} P'_l = \beta'_k \prod_{l \in I'(k)} P'_l  $$
 using Lemma~\ref{mutbetaM}. This is the desired equality. 

  Now consider the case $j=k+1$. Similarly we have 
  $$ I'(k+1) \sqcup \left( J \cap \{ s \} \right)  = I(k) \sqcup \{ k \} . $$
 By Property (B) at the index $k$, we can write
  $$ P'_{k+1}P'_{(k+1)_{-}(\mathbf{i}')} = P_{k+2}P_r =  \frac{P_{k+2}}{P_k} \beta_k \prod_{l \in I(k)} P_l $$
  and thus
 $$ P'_{k+1}P'_{(k+1)_{-}(\mathbf{i}')} = \beta_k \frac{P_{k+2}}{P_k} \frac{P_s}{P'_k} \prod_{l \in I'(k+1)} P'_l = \beta_k \frac{P_{k+2}P_s}{P_k}  \frac{\beta_{k+1}P_k}{\beta_kP_{k+2}P_s} \prod_{l \in I'(k+1)} P'_l  \quad \text{by~\eqref{Aprime}.}  $$
 This simplifies as
 $$ P'_{k+1}P'_{(k+1)_{-}(\mathbf{i}')} = \beta_{k+1} \prod_{l \in I'(k+1)} P'_l  $$
  which is the desired equality as $\beta_{k+1}=\beta'_{k+1}$ by Lemma~\ref{mutbetaM}.
 
 The remaining case to consider is $j=k+2$. 
 We have 
 $$ I'(k+2) \sqcup \{k \} =  I(k+1) \sqcup \{ k+1 \} $$
 and all the indices in $I'(k+2)$ other than $k+1$ are strictly smaller than $k$.  Thus  we have 
$$ \prod_{l \in I'(k+2)} P'_l = P'_{k+1} \prod_{l \in I'(k+2) \setminus \{ k+1 \}} P_l =   \frac{P'_{k+1}}{P_k} \prod_{l \in I(k+1)} P_l = \frac{P_{k+2}}{P_k} \prod_{l \in I(k+1)} P_l . $$
 Thus we have 
  $$ P'_{k+2}P'_{(k+2)_{-}(\mathbf{i}')} = P_{k+1}P'_k = P_{k+1} \frac{\beta_k P_{k+2}P_s}{\beta_{k+1}P_k} = \beta_{k+1} \prod_{l \in I(k+1)} P_l \frac{\beta_kP_{k+2}}{\beta_{k+1}P_k} $$
  using Property (B) at $k+1$. This yields
  $$ P'_{k+2}P'_{(k+2)_{-}(\mathbf{i}')} = \beta_k \prod_{l \in I'(k+2)} P'_l  = \beta'_{k+2} \prod_{l \in I'(k+2)} P'_l $$
  using Lemma~\ref{mutbetaM}. This finishes the proof. 
 \end{proof}
 
 \smallskip
 
  Finally we prove that Property (C) propagates under mutation. The key arguments are provided by Lemma~\ref{Dbarhat} and Lemma~\ref{cormult}.
 
 \begin{prop} \label{mutmult}
 Property (C) holds for the seed $\s^{\mathbf{i}'}$ i.e. one has 
 \begin{equation} \label{Cprime} \tag{C'}
  (\beta'_i ; P'_j) - (\beta'_i ; P'_{j_{+}(\mathbf{i}')}) \leq 1 
   \end{equation}
 for every $1 \leq i \leq N$ and $j \in J_{ex}(\mathbf{i}')$. 
 \end{prop}
 
 \begin{proof}
 First note that $J_{ex}(\mathbf{i}')=J_{ex}(\mathbf{i})$. Note also that there is nothing to prove if $j \notin \{r,s,k,k+1,k+2\}$. Moreover $P'_{k+1} = P_{k+2}$ and $(k+1)_{+}(\mathbf{i}') = (k+2)_{+}(\mathbf{i})$. Thus there is nothing to prove either for $j = k+1$ and similarly for $j = k+2$. We now focus on the cases $j=r, j=s$ and $j=k$. 
 
  Consider the case $j=r$. One has
 $$ (\beta'_i ; P'_r) - (\beta'_i ; P'_{r_{+}(\mathbf{i}')}) = (\beta'_i ; P_r) - (\beta'_i ; P'_{k+1}) = (\beta'_i ; P_r) - (\beta'_i ; P_{k+2}) . $$
   Lemma~\ref{cormult} implies that $(\beta_i ; P_r) = 0$ for every $i>r$. Thus $(\beta'_i ; P_r) = 0$ for every $i>r$ and the desired inequality holds. If $i=r$  then  again Lemma~\ref{cormult} implies $(\beta_i ; P_r) = 1$ and the conclusion is the same as $\beta_r=\beta'_r$. Assume $i<r$. By  Lemma~\ref{Dbarhat}, we have $ \beta_k P_{k+2} P_s = \beta_{k+2} P_r P_{k+1}$. This yields
  $$ (\beta'_i ; P_r) - (\beta'_i ; P_{k+2}) = (\beta_i ; P_r) - (\beta_i ; P_{k+2}) = (\beta_i ; P_s) - (\beta_i ; P_{k+1}) = (\beta_i ; P_s) - (\beta_i ; P_{s_{+}(\mathbf{i})}) . $$
 The desired inequality follows from Property (C) at $j=s$. 
 
  Consider the case $j=s$. One has
  $$ (\beta'_i ; P'_s) - (\beta'_i ; P'_{s_{+}(\mathbf{i}')}) = (\beta'_i ; P_s) - (\beta'_i ; P'_{k}) . $$
  Lemma~\ref{cormult} implies that $(\beta_i ; P_s) = 0$ for every $i>s$. Thus $(\beta'_i ; P_s) = 0$ for every $i>s$ and the desired inequality holds. If $i=s$  then  again Lemma~\ref{cormult} implies $(\beta_i ; P_s) = 1$ and the conclusion is the same as $\beta'_s=\beta_s$. Assume $i<s$. Then we have 
 \begin{align*}
   (\beta'_i ; P_s) - (\beta'_i ; P'_{k}) &=  (\beta_i ; P_s) - (\beta_i ; P'_k)  = (\beta_i ; P_s) - (\beta_i ; P_s P_{k+2}) + (\beta_i ; P_{k})  \\
   &= (\beta_i ; P_{k}) - (\beta_i ; P_{k+2}) = (\beta_i ; P_{k}) - (\beta_i ; P_{k_{+}(\mathbf{i})}) . 
    \end{align*}
  Thus we can conclude using Property (C) at $j=k$. 
  
  Consider the case $j=k$. One has
   $$ (\beta'_i ; P'_k) - (\beta'_i ; P'_{k_{+}(\mathbf{i}')}) = (\beta'_i ; P'_k) - (\beta'_i ; P'_{k+2}) = (\beta'_i ; P'_k) - (\beta'_i ; P_{k+1})  . $$
 Proposition~\ref{secondprop} implies that Equation~\eqref{secondhyp} holds for the seed $\s^{\mathbf{i}'}$ and in particular we can apply Lemma~\ref{cormult} for $\s^{\mathbf{i}'}$ (recall that the statement of Lemma~\ref{cormult} is proved using only Properties (A) and (B)). Therefore  $(\beta'_i ; P'_k) = 0$ if $i>k$ and $(\beta'_k ; P'_k) = (\beta_{k+2} ; P'_k) = 1$. As before we can focus on the case $i<k$. In particular $\beta'_i=\beta_i$. Thus we have 
  \begin{align*}
   (\beta'_i ; P'_k) - (\beta'_i ; P_{k+1}) &=  (\beta_i ; P'_k) - (\beta_i ; P_{k+1}) =  (\beta_i ; P_r P_{k+1}) - (\beta_i ; P_k) - (\beta_i ; P_{k+1})  \\
    & =  (\beta_i ; P_r) - (\beta_i ; P_{k}) =  (\beta_i ; P_r) - (\beta_i ; P_{r_{+}(\mathbf{i})}) 
    \end{align*}
  and the Property (C) at $j=r$ allows us to conclude. 
   This proves that Property~\eqref{Cprime}  holds. 
 \end{proof}
    
 \section{Initial seed in type $A_n$}
 \label{initAn}


 In this section we prove Theorem~\ref{thminitcond} in the case $\mathfrak{g} = \mathfrak{sl}_{n+1}$ where $n \geq 1$ is fixed. 
 We denote by $I = \{1, \ldots , n \}$ the index set of the simple roots and we consider the natural order on $I$ given by $1<2< \cdots < n$. As explained in Section~\ref{reminddeterm}, this yields a convex order on the set $\Phi_{+}$ of positive roots, corresponding to the reduced expression of $w_0$ given by 
 $$ \mathbf{i}_{nat} :=  (1,2,1,3,2,1, \ldots , n,n-1, \ldots , 1) . $$
 The aim of this section is to check that the standard seed $\s^{\mathbf{i}_{nat}}$ satisfies Properties (A), (B) and (C). We use Kang-Kashiwara-Kim-Oh's monoidal categorification of the cluster structure of $\mathbb{C}[N]$ via representations of quiver Hecke algebras. More precisely, the cluster variables of  $\s^{\mathbf{i}_{nat}}$ are categorified by certain determinantal modules in $R-gmod$. These were explicitly described in \cite{Casbi} in  terms of Kleshchev-Ram's dominant words. Let us briefly remind the necessary setting. 
 
 The set $\mathcal{GL}$ of good Lyndon words is given by 
 $$ \mathcal{GL} = \{ (i, i+1, \ldots , j) \mid i,j \in I,  i \leq j \} . $$
 Recall that this set is totally ordered with respect to the lexicographic order induced by the chosen order on $I$. The dominant words parametrizing the simple objects in $R-gmod$ according to Kleshchev-Ram's classification (see Section~\ref{remindKLR3}) are concatenations of elements of $\mathcal{GL}$ in the decreasing order. Dominant words thus coincide with Zelevinsky's multisegments in this case. For any dominant word, we denote by $L(\mu)$ the unique (up to isomorphism and grading shift) simple module associated to $\mu$.
 
 For every $i \leq j$, we will use the notation $[i;j]$ for the positive root $\alpha_i + \cdots + \alpha_j$. The integer $j-i+1$ is called the \textit{height} of this positive root. For each $1 \leq r \leq n$ the occurrences of $r$ in $\mathbf{i}_{nat}$ correspond to positive roots of height $r$. More precisely, for every $1 \leq k \leq n$ the $k$th occurrence of $r$ in $\mathbf{i}_{nat}$ corresponds to the positive root $[k;r+k-1]$. Equivalently for every $i \leq j$, $[i;j]$ is the positive root corresponding to the $i$th occurrence of $j-i+1$ in $\mathbf{i}_{nat}$. 
 
 Now consider the standard seed $\s^{\mathbf{i}_{nat}}$ of $\mathbb{C}[N]$, let $x_1, \ldots, x_N$ denote its cluster variables and $M_1, \ldots , M_N$ denote the corresponding determinantal modules in $R-gmod$. The dominant words associated to $M_1, \ldots, M_N$ were computed in \cite{Casbi} and are given as follows. 
 
 \begin{prop}[{{\cite[Theorem 6.1]{Casbi}}}] \label{initseedAn}
 For every $1 \leq r \leq n$ and $1 \leq k \leq n-r+1$ the determinantal module corresponding to the $k$th occurrence of $r$ in $\mathbf{i}_{nat}$ is 
 $$ L \left( [k;r+k-1] [k-1;r+k-2] \cdots [1;r] \right) . $$
 \end{prop}
 
  \smallskip
  
 For every $0 \leq r \leq n$ and $0 \leq k \leq n-r+1$  we set 
$$ P[k,r] :=  \prod_{1 \leq l \leq k \leq m \leq r+k-1} [l;m] . $$
Note that $P[k,r] = 1$ if $k=0$ or $r=0$. 
We begin by proving that the standard seed $\s^{\mathbf{i}_{nat}}$ satisfies Property (A). More precisely, we show that if $j$ is the position of the $k$th occurrence of the letter $r$ in $\mathbf{i}_{nat}$, then $P_j = P[k,r]$. 
 
 \begin{lem} \label{Dbarinit}
  For any $1 \leq r \leq n$ and $1 \leq k \leq n-r+1$, the determinantal module   $L \left( [k;r+k-1] \cdots [1;r] \right)$ is strongly homogeneous and one has 
  $$ \overline{D} \left( L \left( [k;r+k-1] \cdots [1;r] \right) \right) = \frac{1}{P[k,r]} . $$
In particular the standard seed $\s^{\mathbf{i}_{nat}}$ satisfies Property (A).
  \end{lem}
  
   \begin{proof}
   Let $w[k,r]$ denote the element of $W$ given by 
   $$ w[k,r] : = s_k s_{k+1} \cdots s_{r+k-1} s_{k-1} s_k \cdots s_{r+k-2} \cdots s_1 s_2  \cdots s_r . $$
   It is immediate to check that for any $1 \leq j \leq n$, there is exactly one occurrence of each neighbour of $j$ between two consecutive occurrences of $j$ in the word $(k, k+1, \ldots , r+k-1, k-1, k, \ldots , r+k-2 , \ldots, 1, 2 , \ldots r)$. Moreover the last occurrence of $j$ is either strictly inside the last segment (if $1 \leq j < r$) and in this case there is exactly one occurrence of a neighbour of $j$ (namely $j+1$) after this occurrence, or it is the last letter of one of the segments (if $j \geq r$) in which case there is exactly one occurrence of $j-1$ and no occurrence of $j+1$ after this occurrence. Therefore there cannot be any subword of the form $s_is_js_i$ with $i \cdot j =-1$ in any reduced expression of $w[k,r]$ (i.e. $w[k,r]$ is fully-commutative in the sense of Defintion~\ref{definitionfullcom}) and the chosen reduced expression is reduced. Moreover using Stembridge's results \cite{Stem} (see Theorem~\ref{Stembridge}) we conclude that $w[k,r]$ is dominant minuscule. Hence by the construction of Kleshchev-Ram \cite{KRhom} (see Theorem~\ref{thmKRhom}) the determinantal module $L \left( [k;r+k-1] \cdots [1;r] \right)$ is strongly homogeneous. 
   
   By Proposition~\ref{propbarD}, we have 
$$  \overline{D} \left( L \left( [k;r+k-1] \cdots [1;r] \right) \right) =  \prod_{\beta \in \Phi_{+}^{w[k,r]}} \frac{1}{\beta} . $$
 We prove by induction on $k$ that for any $r \geq 1$, 
$$ \Phi_{+}^{w[k,r]} = \{ [l;m]  , 1 \leq l \leq k \leq m \leq r+k-1 \} . $$
If $k=1$ then  for any $r \geq 1$ one has 
$$ \Phi_{+}^{w[1,r]} =  \Phi_{+}^{s_1 s_2  \cdots s_r} = \{ \alpha_1 , \alpha_1 + \alpha_2 , \ldots , \alpha_1 + \cdots + \alpha_r \} = \{ [1;m] , 1 \leq m \leq r \} $$
which is the desired equality for $k=1$. Assume the result holds at rank $k-1$. It is straightforward to check that for every $l,m$ such that $1 \leq l \leq k-1 \leq m \leq r+k-2$, we have 
$$ s_k s_{k+1} \cdots s_{r+k-1}([l;m])  = [l;m+1] . $$
Applying the induction hypothesis, this implies 
$$ s_k s_{k+1} \cdots s_{r+k-1} \left( \Phi_{+}^{w[k-1,r]} \right) =  \{ [l;m+1]  , 1 \leq l \leq k-1 \leq m \leq r+k-2 \} . $$
Thus we get 
\begin{align*}
 \Phi_{+}^{w[k,r]} &= \{ [k] , [k;k+1] , \ldots, [k;r+k-1] \} \sqcup s_k s_{k+1} \cdots s_{r+k-1} \left( \Phi_{+}^{w[k-1,r]} \right)  \\
 &= \{ [k;m] , k \leq m \leq r+k-1  \} \sqcup  \{ [l;m+1]  , 1 \leq l \leq k-1 \leq m \leq r+k-2 \}  \\
  &=  \{ [k;m] , k \leq m \leq r+k-1  \} \sqcup   \{ [l;m]  , 1 \leq l \leq k-1 < m \leq r+k-1 \}  \\
  &=  \{ [l;m]  , 1 \leq l \leq k \leq m \leq r+k-1 \} . 
 \end{align*}
This finishes the proof. 
   \end{proof}
   
   \begin{cor}
   The standard seed $\s^{\mathbf{i}_{nat}}$ satisfies Property (C). 
      \end{cor}
      
       \begin{proof}
   It follows from the previous Lemma that the multiplicity of any positive root in any of the polynomials $P_j , 1 \leq j \leq N$ is always equal to $0$ or $1$. In particular the Property (C) is trivially satisfied. 
       \end{proof}
   
   \smallskip
   
 Now we check that Property (B) is also satisfied.  For each $j \in \{1 , \ldots , N \}$, the polynomials involved in the right hand-side of Property (B) at the label $j$ correspond to the last occurrences in $\mathbf{i}_{nat}$ of each neighbour of $i_j$ (strictly) before the position $j$. Thus if $i_j := r$, we have to consider the last occurrences of $r+1$ and $r-1$ before $j$. 
   
   \begin{lem}   
   The standard seed $\s^{\mathbf{i}_{nat}}$ satisfies Property (B). 
   \end{lem}
 
 \begin{proof}
 
Fix $1 \leq r \leq n$ and $1 \leq k \leq n-r+1$  and let $j \in \{1, \ldots , N \}$ denote the position of the $k$-th occurrence of $r$ in $\mathbf{i}_{nat}$. By definition $j_{-}$ corresponds to the $(k-1)$-th occurrence of $r$. Hence  by Proposition~\ref{initseedAn} together with Lemma~\ref{Dbarinit}, we have $P_j=P[k,r]$ and $P_{j_{-}}=P[k-1,r]$ (note that $j_{-}=0$ if and only if $k=1$ and in this case we have $P_0=P[0,r]=1$). 
It is not hard to see that if $r \leq n-1$ then the letter $r+1$ has appeared exactly $k-1$ times before $j$ in $\mathbf{i}_{nat}$ as it only appears in the subwords of $\mathbf{i}_{nat}$ of the form $s, \ldots , 1 $ with $s > r$. Hence the last occurrence of $r+1$ before $j$ in $\mathbf{i}_{nat}$ corresponds to the $k-1$th occurrence of $r+1$ in $\mathbf{i}_{nat}$. Thus by Proposition~\ref{initseedAn} the associated determinantal module is
$$ L \left( [k-1;r+k-1] [k-2;r+k-2] \cdots [1;r+1] \right) . $$
 Similarly one can check that if $r \geq 2$ then the letter $r-1$ appears exactly $k$ times before $j$ in $\mathbf{i}_{nat}$. Therefore the last occurrence of $r-1$ before $j$ corresponds to the $k$ occurrence of $r-1$ and Proposition~\ref{initseedAn} implies that the associated determinantal module is 
 $$  L \left( [k;r+k-2] [k-1;r+k-3] \cdots [1;r-1] \right) . $$
We can now rewrite Property (B) at the label $j$ as 
$$  P_j P_{j_{-}} = [k;r+k-1]  P[k,r-1]  P[k-1,r+1]  $$
for any $1 \leq r \leq n$ and $1 \leq k \leq n-r+1$.
 We have 
\begin{align*}
 P_j P_{j_{-}} & =  P[k,r] P[k-1,r] \\
  &= \prod_{1 \leq l \leq k \leq m \leq r+k-1} [l;m] \quad  \prod_{1 \leq l \leq k-1 \leq m \leq r+k-2} [l;m]   \\
   &=   \left( \prod_{1 \leq l \leq k} [l;r+k-1]  \prod_{1 \leq l \leq k \leq m \leq r+k-2} [l;m]  \right) \\
    & \qquad \qquad \qquad   \left(  \left( \prod_{1 \leq l \leq k-1} [l;r+k-1] \right)^{-1}  \prod_{1 \leq l \leq k-1 \leq m \leq r+k-1} [l;m]  \right) \\
    &= [k;r+k-1] \prod_{1 \leq l \leq k \leq m \leq r+k-2} [l;m] \quad  \prod_{1 \leq l \leq k-1 \leq m \leq r+k-1} [l;m]   \\
     &= [k;r+k-1]  P[k,r-1]  P[k-1,r+1] .
 \end{align*}
This proves that Property (B) is satisfied. 
  \end{proof}

\section{Initial seed in type $D_4$}
\label{initD4}

  This section is devoted to the proof of Theorem~\ref{thminitcond} when $\mathfrak{g}$ is of type $D_4$. In this case, determinantal modules are not necessarily homogeneous. Hence we cannot always use the results of \cite{KRhom,Nakada} to compute the images under $\overline{D}$ of the flag minors. Therefore, the most difficult part is to prove that Property (A) is satisfied for a certain standard seed. Properties (B) and (C) will then be rather straightforward to check. 
   
   \bigskip
   
  We fix the natural ordering on the set of vertices of its Dynkin diagram as in {{\cite[Section 8.7]{KR}}}, i.e. $1<2<3<4$ with $3$ being the trivalent node. There are twelve positive roots and hence also twelve cluster variables in every seed, with four frozen variables and eight unfrozen variables. The good Lyndon words (as well as the corresponding cuspidal representations in $R-gmod$) can be found in {{\cite[Section 8.7]{KR}}} or can be directly computed using the algorithm described by Leclerc {{\cite[Section 4.3]{Leclerc}}}. Let us range them in the increasing order:
  $$ \mathcal{GL} = \{ 1 < 13 < 132 < 134 < 1342 < 13423 < 2 < 23 < 234 < 3 < 34 < 4 \} . $$  
 Thus the corresponding  convex order on the set of positive roots is given by
 \begin{align*}
  \Phi_{+} = \{ & \alpha_1   < \alpha_1 + \alpha_3   <  \alpha_1 + \alpha_3 + \alpha_2  <   \alpha_1 + \alpha_3 + \alpha_4  < \alpha_1 + \alpha_2 +  \alpha_3 + \alpha_4  \\
   & < \alpha_1 + 2 \alpha_3 + \alpha_2 + \alpha_4 < \alpha_2 < \alpha_2 + \alpha_3 < \alpha_2 + \alpha_3 + \alpha_4 < \alpha_3 < \alpha_3 + \alpha_4 < \alpha_4 \} . 
  \end{align*}
  The corresponding reduced expression of $w_0$ is $\mathbf{i}_{nat} = (1,3,2,4,3,1,4,3,2,4,3,4)$. Using Theorem~\ref{thmcasbi2}, we find the following dominant words for the determinantal modules of the seed ${\s}^{\mathbf{i}_{nat}}$:
  $$ 1 , 13 , 132 , 134 , 134213 , 134231 , 2134 , 23134213 , 234132 , 32134 , 3423134213 , 432134 . $$
 Among these determinantal modules, those whose dominant words are  $134231, 234132,  \allowbreak  3423134213$ and $432134$ correspond to the frozen variables in $\mathbb{C}[N]$. Three of them turn out to be strongly homogeneous but one of them is not homogeneous, namely $L(3423134213)$. From now on we denote this module by $M$.

  \subsection{Computation of a graded character}
  
  In this subsection we determine the whole graded character of $M$, as an intermediate step for the computation of $\overline{D}([M])$. 
  
 As the isomorphism class of $M$ is a frozen variable, it follows from the monoidal categorification results of Kang-Kashiwara-Kim-Oh \cite{KKKO} that $M$ $q$-commutes with every simple module in $R-gmod$. In particular, it commutes with all the cuspidal representations in $R-gmod$. This strong property constrains the form of the graded character of this module. Actually, it is sufficient to use the fact that $M$ $q$-commutes with the four cuspidal modules corresponding to the simple roots $\alpha_i , i \in \{1, \ldots , 4 \}$. First one needs to determine the homogeneous degrees of the renormalized $R$-matrices $r_{M , L(i)}$ for every $i \in \{1, \ldots , 4 \}$. We denote these degrees by $\Lambda \left( M , L(i) \right)$ and $\Lambda \left( L(i) , M  \right)$ following \cite{KKKO}. The commuting of $M$ and $L(i)$ yields an isomorphism of graded modules 
 \begin{equation} \label{isocomut}
  M \circ L(i) \simeq q^{\Lambda (L(i),M)} L(i) \circ M 
  \end{equation}
 for every $i \in I$. 
 
  \begin{lem} \label{lemcomut}
   Let $M := L(3423134213)$. Then one has 
  $$ \Lambda \left( M , L(i) \right) = \Lambda \left( L(i) , M  \right) = 0 $$
  for every $1 \leq i \leq 4$,
  \end{lem}
  
   \begin{proof}
 We let $\beta := \wt(M) = 2 \left( \alpha_1 + \alpha_2 + 2 \alpha_3 + \alpha_4 \right)$.  It follows from the definition of $\Lambda(M,N)$ (see \cite{KKK,KKKO}), that if $N \in R(\gamma)-gmod$ with $ \gamma \in Q_{+}$ then 
   $$ \Lambda(M,N) = - (\beta , \gamma) + 2(\beta , \gamma)_{n}  - 2s_{M,N} $$
   where $s_{M,N}$ is the largest non-negative integer $s$ such that the image of $R_{M,N_z}$ is contained in $z^{s}(N_z \circ M)$. First consider $i \in \{1,2,4\}$. Then one has $(\alpha_i , \wt(M)) = 0$. Moreover, both $s_{L(i),M}$ and $s_{M,L(i)}$ are always smaller or equal to the number of occurrences of $i$ in $3423134213$ which is by definition $(\alpha_i , \wt(M))_{n}$. Hence one has 
   $$ \Lambda \left( M , L(i) \right) \geq 0 \quad \text{and} \quad  \Lambda \left( L(i) , M \right) \geq 0 . $$
 By {{\cite[Lemma 3.2.3]{KKKO}}}, the commuting of $M$ and $L(i)$ exactly means that the sum of these two quantities has to be zero. Hence one has $s_{L(i),M} = s_{M,L(i)} = 0$ and 
 $$ \Lambda \left( M , L(i) \right) =  \Lambda \left( L(i) , M \right) = 0 . $$
   For $i=3$ one can perform computations of $s_{M,L(i)}$ and $s_{L(i),M}$ in the following way (analogous to {{\cite[Section 6D]{Casbi}}}, see also {{\cite[Corollary 6.6]{Casbi}}} and {{\cite[Remark 6.7]{Casbi}}}). Recall that $M = \text{hd} \left( L(34) \circ L(23) \circ L(1342) \circ L(13) \right)$. Choose non zero vectors $v_{34}, v_{23}, v_{1342}, v_{13}$ respectively generating the four cuspidal modules involved in the brackets. Similarly choose $v_3$ a generating vector of $L(3)_z := \mathbb{C}[z] \otimes_{\mathbb{C}} L(3)$ on which the actions of the generators of the quiver Hecke algebras are given (for example) in {{\cite[Section 2.2]{KKKO}}}. The integer $s:=s_{M,L(3)}$ is given by the valuation of the polynomial 
   $$ (\tau_1(x_1-x_2)+1) \tau_2 \tau_3 (\tau_4(x_4-x_5)+1) \tau_5 (\tau_6(x_6-x_7)+1)\tau_7 \tau_8 \tau_9 (\tau_{10}(x_{10}-x_{11}) + 1) \cdot v $$
   where  $v := v_{34} \otimes v_{23} \otimes v_{1342} \otimes v_{13} \otimes v_3$.
  Looking at the characters of the modules  $L(34), L(23), L(1342), L(13)$ recalled above, we see that $\tau_9$ (resp. $\tau_5$, $\tau_3$) acts trivially on $v_{13}$ (resp. $v_{1342}$, $v_{23}$) because the weight $31$ (resp. $3142$, $32$) does not appear in the character of $L(13)$ (resp. $L(1342)$, $L(23)$). Hence the above quantity is equal to 
  \begin{align*}
  -z^3 (\tau_1(x_1-x_2)+1) \tau_2 \cdots \tau_{10} \cdot v 
 = ( z^4 \tau_1 \cdots \tau_{10}  - z^3 \tau_2 \cdots \tau_{10}) \cdot v .
  \end{align*}
  Therefore we have $s=3$ i.e. $\Lambda(M,L(3))=0$ and hence $\Lambda(L(3),M)=0$ as well by what precedes. 
   \end{proof}
   
   Writing the $q$-commutations relations with every $L(i)$ and looking at the coefficients in front of every weight in these relations, one can compute the entire graded character of $M$. 
  
   \begin{prop} \label{qcharfrozen}
   The graded character of the frozen variable $L(3423134213)$ is given by:
   \begin{itemize}
   \item The following weights appear with $q$-dimension $1$:
   \begin{align*}
& 3\{2,4\}313\{1,2,4\}3 \quad  3\{1,4\}323\{1,2,4\}3 \quad 3\{2,1\}343\{1,2,4\}3 \\ & \quad 3\{1,2,4\}313\{2,4\}3  \quad 3\{1,2,4\}323\{1,4\}3 \quad 3\{1,2,4\}343\{2,1\}3. 
 \end{align*} 
 \item The following weights appear with $q$-dimension $q+q^{-1}$:
  $$ 3\{2,4\}3113\{2,4\}3 \quad  3\{2,1\}3443\{2,1\}3 \quad  3\{1,4\}3223\{1,4\}3  \quad 3\{1,2,4\}33\{1,2,4\}3. $$
   \end{itemize}
\end{prop} 

 In this statement, the notation $\{1,2,4\}$ means any of the six permutations of $1,2$ and $4$. Similarly $\{1,4\}$ means any of the words $14$ or $41$.

 \begin{proof}
 As above we set $M := L(3423134213)$. By definition, one has 
 $$ M = \text{hd} \left( L(34) \circ L(23) \circ L(1342) \circ L(13) \right) $$
 where $L(34), L(23), L(1342), L(13)$ are the cuspidal representations corresponding respectively to the positive roots $\alpha_3 + \alpha_4, \alpha_2 + \alpha_3, \alpha_1 + \alpha_2 + \alpha_3 + \alpha_4, \alpha_1 + \alpha_3$ for the natural ordering of the type $D_4$ Dynkin diagram. The graded characters of these four modules are known from Kleshchev-Ram {{\cite[Section 8.7]{KR}}}. They are respectively given by $(34),  (23), (1342) + (1324), (13)$ . 
 In particular the weights appearing in the graded character of $M$ are either of the form  $(34) \shuffle (23) \shuffle (1342) \shuffle (13)$ or of the form $(34) \shuffle (23) \shuffle (1324) \shuffle (13)$ (Recall the notation $\shuffle$ from Section~\ref{remindKR3}). We write 
 $$ ch_q(M) = \sum_{\mathbf{j}} P_{\mathbf{j}}(q) \mathbf{j} $$
 with $P_{\mathbf{j}}(q) \in \mathbb{Z}_{\geq 0}[q^{\pm 1}]$ for every $\mathbf{j}$. 
The graded character of the cuspidal module $L(i)$ is simply $(i)$ for every $i \in I$. Thus for each $i \in I$ the graded isomorphisms~\eqref{isocomut}  can be written in terms of quantum shuffle products of graded characters as
\begin{equation} \label{equalchar}
   (i) \circ ch_q(M) = ch_q(M) \circ (i) 
   \end{equation}
  as the degrees of each of the corresponding $R$-matrix is zero by the previous lemma. We now show how these four equations strongly constrain the weights of $ch_q(M)$ as well as their coefficients in $\mathbb{Z}_{\geq 0}[q^{\pm 1}]$. 
  
  \smallskip
  
  \textbf{All the weights of $M$ begin with $3$.}
  First, the first letter of a weight of $M$ is necessarily the first letter of one of the words $34,23,1342, 1324, 13$ and in particular it cannot be $4$. 
 If $\mathbf{j}$ is a weight of $M$ beginning with $2$, we write $\mathbf{j} = 2 . \mathbf{j}'$. Consider the Equation~\eqref{equalchar} for $i=2$. The word $2 . 2 . \mathbf{j}'$ appears in both hand sides and can only come from the shuffle of $(2)$ with $\mathbf{j}$ as there are no shuffle of $34,23,1342,13$ (or $34,23,1324,13$) beginning with $2,2$. Hence the quantum shuffle formula (see Proposition~\ref{quantumshuffle}) yields
 $$ (1 + q^{-2}) P_{2 . \mathbf{j}'} (q) =  (1 + q^{2}) P_{2 . \mathbf{j}'} (q) $$
 i.e.  $P_{\mathbf{j}}(q) =0$. 
 
   If $\mathbf{j}$ is a weight of $M$ beginning with $1$, we write $\mathbf{j} = 1 . \mathbf{j}'$. Let us first prove that $\mathbf{j}'$ cannot begin with $1$. Assume so and write $\mathbf{j}' = 1 . \mathbf{j}''$ and thus $\mathbf{j} = 1 . 1. \mathbf{j}''$. Consider the Equation~\eqref{equalchar} for $i=1$. The word $1 . 1 . 1 .  \mathbf{j}''$ appears in both hand sides and can only come from the shuffle of $(1)$ with $\mathbf{j}$ as there are no shuffle of $34,23,1342,13$ (or $34,23,1324,13$) beginning with $1,1,1$. Hence the quantum shuffle product formula yields
 $$ (1 + q^{-2} + q^{-4}) P_{1 . 1 . \mathbf{j}''} (q) =  (1 + q^{2} + q^{4}) P_{1 . 1 . \mathbf{j}''} (q) $$
 i.e.  $P_{\mathbf{j}}(q) =0$. This proves that $\mathbf{j}'$ does not begin with $1$, i.e. there are no weights of $M$ beginning with $1,1$. Hence applying the same argument as for $i=2$ yields $P_{\mathbf{j}}(q) = P_{1 . \mathbf{j}'}(q) =0$. 
 
  \smallskip
 
 \textbf{There are at least two letters between the first two occurrences of $3$.}
 Assume $\mathbf{j}$ is a weight of $M$ of the form $3,j,3 . \mathbf{j}'$ (with $j \neq 3$). Consider the Equation~\eqref{equalchar} for $i=3$. The word $3,j,3,3 . \mathbf{j}'$ appears in both hand sides and can only come from the shuffle of $(3)$ with $ 3,j,3 . \mathbf{j}' = \mathbf{j}$ as there are no shuffle of $34,23,1342,13$ (or $34,23,1324,13$) beginning with $3,j,3,3$ and we have  just proved that there are no weights of $M$ beginning with $j$. Hence the quantum shuffle product formula yields
 $$(q^{-1} + q^{-3}) P_{\mathbf{j}}(q) = (q + q^{-1}) P_{\mathbf{j}}(q) $$
 i.e. $P_{\mathbf{j}}(q) =0$. 
 
  \smallskip
 
 \textbf{There is at most one occurrence of each letter $1,2,4$ between the first two occurrences of $3$.}
  We begin by showing that if $\mathbf{j}$ is a weight of $M$ of the form $3 . \mathbf{k} . i,i . \mathbf{l}$ with $i \neq 3$ and  $\mathbf{k}$ not containing $3$ then $P_{\mathbf{j}}(q)=0$. Indeed consider the Equation~\eqref{equalchar} for $i$. The word $3 . \mathbf{k} . i,i,i . \mathbf{l}$ appears in both hand sides and can only come from the shuffle of $(i)$ with $3 . \mathbf{k} . i,i . \mathbf{l} =\mathbf{j}$. The quantum shuffle formula yields
 $$ (q + q^{-1} + q^{-3}) P_{\mathbf{j}}(q) = (q^3 + q + q^{-1}) P_{\mathbf{j}}(q)$$
 which proves $P_{\mathbf{j}}(q) = 0$ in this case. 
 
  Now assume $\mathbf{j}$ is a weight of $M$ of the form $3 . \mathbf{j}' . 3 . \mathbf{j}''$ where $\mathbf{j}'$ contains two occurrences of a letter $i \in \{1,2,4\}$ and does not contain $3$; as there are exactly two occurences of every letter $1,2,4$ and four occurrences of $3$ in any weight of $M$ (in particular in $\mathbf{j}$), it follows that $\mathbf{j}''$ does not contain any occurrence of this letter $i$ and contains two occurrences of $3$. Write $\mathbf{j}' = \mathbf{k}_1 . i . \mathbf{k}_2 . i \mathbf{k}_3$ with $\mathbf{k}_1, \mathbf{k}_2, \mathbf{k}_3$ containing neither $i$ nor $3$ and consider the Equation~\eqref{equalchar} for $i$. The word $3 . \mathbf{k}_1 . i .  \mathbf{k}_2 . i,i . \mathbf{k}_3 . 3 . \mathbf{j}''$ appears in both hand sides and can only come from the shuffle of $(i)$ with  $\mathbf{j}$ as we have just proved that there are no weights of $M$ of the form $3 . \mathbf{k}_1 . \mathbf{k}_2 . i,i .\mathbf{k}_3 . 3 .\mathbf{j}''$. Hence the quantum shuffle formula yields
  $$ (q^{-1} + q^{-3}) P_{\mathbf{j}}(q)  = (q^3 + q) P_{\mathbf{j}}(q) $$
  i.e. $P_{\mathbf{j}}(q) = 0$.
 
 This proves that the weights of $M$ begin either with $3 \{1,2\} 3$ or $3 \{1,4\} 3$ or $3 \{2,4\} 3$ or $3 \{1,2,4\} 3$. The exact same arguments can be applied in a symmetric way to successively prove that 
 \begin{enumerate}
 \item All the weights of $M$ end with $3$.
 \item There are at least two letters between the last two occurrences of $3$.
 \item There is at most one occurrence of each letter $1,2,4$ between the last two occurrences of $3$.
 \end{enumerate}
 
Therefore we can write any weight $\mathbf{j}$ of $M$ as 
$$ w = 3 . \mathbf{k}_1 . 3 . \mathbf{k}_2 . 3 . \mathbf{k}_3 . 3 $$
with $u_1$ and $u_3$ either of the form $\{1,2\}$ or $\{1,4\}$ or $\{2,4\}$ or $\{1,2,4\}$ and $u_2$ is a word containing at most two letters (necessarily $1,2$ or $4$). 

 \smallskip

 \textbf{If $\mathbf{k}_2$ contains a letter $i \in \{1,2,4\}$, then $\mathbf{k}_1$ and $\mathbf{k}_3$ are either of the form $\{1,2,4\}$ or $\{j,k\}$ with $j,k \neq i$.}
 Assume $\mathbf{k}_2$ contains $i$ and for example $\mathbf{k}_1$ of the form $\{i,j\}$ ($j \neq i$). Then $i$ does not appear in $\mathbf{k}_3$ as there are only two occurrences of $i$ in $\mathbf{j}$.  As $\mathbf{k}_2$ is of length at most $2$, $i$ is necessarily the first or the last letter of $\mathbf{k}_2$ (or both if $\mathbf{k}_2$ contains only $i$). Assume for example $\mathbf{k}_2$ begins with $i$ and consider the Equation~\eqref{equalchar} for $i$. The word $3 . \mathbf{k}_1 . 3 . i . \mathbf{k}_2 . 3 . \mathbf{k}_3 . 3$ appears in both hand sides and as there is no weight of $M$ with only one single letter between the first two occurrences of $3$, the quantum shuffle formula yields
 $$ (1 + q^{-2}) P_{\mathbf{j}}(q) = (q^2 + 1)P_{\mathbf{j}}(q) $$
and thus $P_{\mathbf{j}}(q) = 0$ which proves the desired statement. It follows in particular that $\mathbf{k}_2$ cannot be composed of two distinct letters (in that case, both $\mathbf{k}_1$ and $\mathbf{k}_3$ could be only of the form $\{1,2,4\}$ but then $\mathbf{j}$ would contain three occurrences of a letter $1,2$ or $4$). It is then straightforward that the only remaining possible weights are the ones listed in the statement of Proposition~\ref{qcharfrozen}. 
 
  \bigskip
  
 To finish the proof of Proposition~\ref{qcharfrozen}, it remains to compute the values of $P_{\mathbf{j}}(q)$ for every weight $\mathbf{j}$ in this list. The starting point is that by {{\cite[Theorem 7.2 (ii)]{KR}}}, we know that the weight space of $M$ corresponding to the highest weight has $q$-dimension $1$. From this we can deduce the $q$-dimensions of all the other weight spaces of $M$. 
 
  \smallskip
 
  \textbf{For every distinct $i,j \in \{1,2,4\}$, one has $P_{3,i,j,3 . \mathbf{j}'}(q) = P_{3,j,i,3 . \mathbf{j}'}(q)$ for any $\mathbf{j}'$.}
  By what precedes, there is no weight of $M$ of the form $3,i,j,i,3 . \mathbf{j}'$. Hence considering the Equation~\eqref{equalchar} for $i$, the word $3,i,j,i,3 . \mathbf{j}'$ appears in both hand sides only as a shuffle of $i$ with $3,i,j,3 . \mathbf{j}'$ or $3,j,i,3 . \mathbf{j}'$. We get
  $$ q^{-1} P_{3,i,j,3 . \mathbf{j}'}(q) + q P_{3,j,i,3 . \mathbf{j}'} = q P_{3,i,j,3 . \mathbf{j}'}(q) + q^{-1} P_{3,j,i,3 . \mathbf{j}'}(q) $$
  which is the desired statement. 
  Using similar arguments one can also show that for every permutation $\sigma$ of the set $\{1,2,4\}$ one has 
  $$ P_{3,1,2,4,3 . \mathbf{j}'}(q) = P_{3, \sigma(1), \sigma(2), \sigma(4), 3 . \mathbf{j}'}(q) $$
  for any $\mathbf{j}'$. The symmetric statements are also valid i.e. 
   
    \begin{enumerate}
 \item For every distinct $i,j \in \{1,2,4\}$, one has $P_{ \mathbf{j}'' . 3,i,j,3}(q) = P_{ \mathbf{j}'' . 3,j,i,3}(q)$ for any $\mathbf{j}''$.
 \item For every permutation $\sigma$ of $\{1,2,4\}$ one has $ P_{ \mathbf{j}'' . 3,1,2,4,3}(q) = P_{ \mathbf{j}'' . 3, \sigma(1), \sigma(2), \sigma(4), 3}(q) $ for any $\mathbf{j}''$.   
 \end{enumerate}
 
  In particular applying this with the highest weight $3423134213$ we get that  all the weights of the form $3\{2,4\}313\{1,2,4\}3$ are of $q$-dimension $1$. Then  one can apply Equation~\eqref{equalchar} for $1$ and look for words of the form  $3\{2,4\}3113\{1,2,4\}3$. The quantum shuffle formula yields
 \begin{align*}
   & (q^2 + 1) P_{3\{2,4\}313\{1,2,4\}3}(q) + q^{-1} P_{3\{2,4\}3113\{2,4\}3}(q)  \\
   & \qquad \qquad \qquad = (1 + q^{-2}) P_{3\{2,4\}313\{1,2,4\}3}(q) + q P_{3\{2,4\}3113\{2,4\}3}(q) . 
     \end{align*}
   Knowing that $P_{3\{2,4\}313\{1,2,4\}3}(q) = 1$, this implies 
   $$ P_{3\{2,4\}3113\{2,4\}3}(q) = q + q^{-1} . $$
  Then looking for words of the form $3\{1,2,4\}3113\{2,4\}3$ in Equation~\eqref{equalchar} for $1$ and applying the quantum shuffle formula, we get
  $$ P_{3\{1,2,4\}313\{2,4\}3}(q) = 1 . $$
  Then considering Equation~\eqref{equalchar} for $3$, one can look for words of the form $3 \{ 1,2,4\}3313 \{2,4\}3$ and we obtain 
  $$ P_{ 3 \{ 1,2,4\}331 \{2,4\}3} = q + q^{-1} . $$ 
 The nodes $1,2$ and $4$ play symmetric roles, one can then use Equation~\eqref{equalchar} for $2$ and then $4$ to deduce the $q$-dimensions of the remaining weight spaces listed in the statement of Proposition~\ref{qcharfrozen}. This finishes the proof. 
  \end{proof}

 \subsection{Computations on equivariant multiplicities}
 
 We can now finish the proof of Theorem~\ref{thminitcond} in type $D_4$. We let $x_1, \ldots , x_{12}$ denote the  flag minors of the seed $\s^{\mathbf{i}_{nat}}$ and $M_1, \ldots , M_{12}$ the corresponding determinantal modules in $R-gmod$. The dominant words $\mu_1, \ldots , \mu_{12}$ associated to $M_1, \ldots , M_{12}$ were listed at the beginning of this Section. Using Theorem~\ref{Stembridge}, it is straightforward to check that the words $\mu_i , i \notin \{5,8,11\}$ are reduced expressions of dominant minuscule elements of $W$. Therefore the corresponding modules are strongly homogeneous and Proposition~\ref{propbarD} implies that $\overline{D}(x_i)$ is of the form $1/P_i$ if $i \notin \{5,8,11\}$. Thus it remains to check that it is also the case for $x_5, x_8$ and $x_{11}$.

 Evaluating at $q=1$ the graded dimensions given in  Proposition~\ref{qcharfrozen}, we can use Equation~\eqref{Dbar} to deduce the equivariant multiplicity of $[M]=x_{11}$. This is done using the formal calculation software SAGE.  
 
 \begin{cor} \label{Dbarfrozen}
 The equivariant multiplicity of the frozen variable $x_{11}=[L(3423134213)] \in \mathbb{C}[N]$ is given by:
  $$ \frac{1}{ \left( \prod_{J \subset \{1,2,4\}} \left( \alpha_3 + \sum_{j \in J} \alpha_j \right) \right) (\alpha_1 + \alpha_2 + 2 \alpha_3 + \alpha_4)^{2}} . $$ 
  \end{cor}
  
  Using similar arguments, one can also compute the equivariant multiplicities of the two remaining flag minors whose corresponding determinantal modules are not homogeneous, namely $x_5=[L(134213)]$ and  $x_8=[L(23134213)]$.  
  
  \begin{lem} \label{Dbarunfrozen}
  The equivariant multiplicity of $x_5=[L(134213)] \in \mathbb{C}[N]$ is given by:
$$  \frac{1}{{\alpha_1}^{2}(\alpha_1+\alpha_3)(\alpha_1+\alpha_2+\alpha_3)(\alpha_1+\alpha_4+\alpha_3)(\alpha_1+\alpha_2+\alpha_3+\alpha_4)} . $$
The equivariant multiplicity of $x_8=[L(23134213)] \in \mathbb{C}[N]$ is given by:
$$  \frac{1}{\alpha_1 \alpha_2 (\alpha_1+\alpha_3)(\alpha_2+\alpha_3)(\alpha_1+\alpha_2+\alpha_3)^{2}(\alpha_1+\alpha_2+\alpha_3+\alpha_4)(\alpha_1+\alpha_2+ 2 \alpha_3+\alpha_4)} . $$
\end{lem}

\smallskip

 This proves that the standard seed $\s^{\mathbf{i}_{nat}}$ satisfies the Property (A). The polynomials $P_j , 1 \leq j \leq 12$ are given by:
 $$
  \begin{gathered}
 P_1 = \alpha_1 \qquad P_2 = \alpha_1(\alpha_1 + \alpha_3) \\
  P_3 = \alpha_1(\alpha_1 + \alpha_3)(\alpha_1 + \alpha_2 + \alpha_3)  \qquad P_4 = \alpha_1(\alpha_1 + \alpha_3)(\alpha_1 + \alpha_3 +  \alpha_4) \\  
 P_5 = {\alpha_1}^{2}(\alpha_1+\alpha_3)(\alpha_1+\alpha_2+\alpha_3)(\alpha_1+\alpha_4+\alpha_3)(\alpha_1+\alpha_2+\alpha_3+\alpha_4) \\
 P_6 = \alpha_1(\alpha_1 + \alpha_3)(\alpha_1 + \alpha_3 +  \alpha_4)(\alpha_1 + \alpha_2 + \alpha_3)(\alpha_1+\alpha_2+\alpha_3+\alpha_4)(\alpha_1+\alpha_2+ 2 \alpha_3+\alpha_4) \\
 P_7 = \alpha_2 \alpha_1 (\alpha_1+\alpha_2+\alpha_3)(\alpha_1+\alpha_2+\alpha_3+\alpha_4) \\
 P_8 = \alpha_1 \alpha_2 (\alpha_1+\alpha_3)(\alpha_2+\alpha_3)(\alpha_1+\alpha_2+\alpha_3)^{2}(\alpha_1+\alpha_2+\alpha_3+\alpha_4)(\alpha_1+\alpha_2+ 2 \alpha_3+\alpha_4) \\
 P_9 = \alpha_2(\alpha_2+\alpha_3)(\alpha_2+\alpha_3+\alpha_4)(\alpha_1+\alpha_2+\alpha_3)(\alpha_1+\alpha_2+\alpha_3+\alpha_4)(\alpha_1+\alpha_2+ 2 \alpha_3+\alpha_4) \\
 P_{10} = \alpha_3(\alpha_2+\alpha_3)(\alpha_1+\alpha_3)(\alpha_1+\alpha_2+\alpha_3)(\alpha_1+\alpha_2+ 2 \alpha_3+\alpha_4) \\
 P_{11} = \left( \prod_{J \subset \{1,2,4\}} \left( \alpha_3 + \sum_{j \in J} \alpha_j \right) \right) (\alpha_1 + \alpha_2 + 2 \alpha_3 + \alpha_4)^{2} \\
 P_{12} =  \alpha_4(\alpha_3+\alpha_4)(\alpha_2+\alpha_3+\alpha_4)(\alpha_1+\alpha_3+\alpha_4)(\alpha_1+\alpha_2+\alpha_3+\alpha_4)(\alpha_1+\alpha_2+ 2 \alpha_3+\alpha_4) . 
 \end{gathered}
 $$
Let us now write down the equalities required by the Property (B). For example, let us detail the cases of  $j=6$ and $j=8$. 
 
 The positive root $\beta_6$ is $\alpha_1+\alpha_2+ 2 \alpha_3+\alpha_4$. Moreover, $i_6= 1$, $6_{-}=1$ and $6_{+}=N+1=13$ as there is no occurrence of the letter $1$ after the position $6$. The node $1$ is monovalent in the Dynkin diagram of type $D_4$ and its only neighbour is $3$. The last occurrence of $3$ in $\mathbf{i}_{nat}$ before the position $6$ is in position $5$.
 Thus Property (B) can be written as
 $$ P_6 P_1 = (\alpha_1+\alpha_2+ 2 \alpha_3+\alpha_4) P_5 . $$
 The positive root $\beta_8$ is $\alpha_2+ \alpha_3$. Moreover, $i_8= 3$, $8_{-}=5$ and $8_{+}=11$. The node $3$ is the trivalent node. The last occurrences before the position $8$ of each of its neighbours are thus $3$ (for the node $2$), $6$ (for the node $1$) and $7$ (for the node $4$). 
 Hence the Property (B) can be written as
 $$ P_8 P_5 = (\alpha_2+ \alpha_3) P_2 P_6 P_7 . $$
The other equalities can be obtained in the same way and are listed below:
$$  
\begin{gathered}
  P_1 = \alpha_1 \qquad P_2 = (\alpha_1 + \alpha_3)P_1 \qquad P_3 = (\alpha_1+\alpha_2+\alpha_3)P_2 \\
   P_4 = (\alpha_1 + \alpha_3 +  \alpha_4) P_2  \qquad P_5 P_2 = (\alpha_1+\alpha_2+\alpha_3+\alpha_4) P_1 P_3 P_4 \\
  P_6 P_1 = (\alpha_1+\alpha_2+ 2 \alpha_3+\alpha_4) P_5  \qquad P_7 P_4  = \alpha_2 P_5 \\
   P_8 P_5 = (\alpha_2+ \alpha_3) P_2 P_6 P_7  \qquad P_9 P_3 = (\alpha_2 + \alpha_3 +  \alpha_4) P_8 \\
P_{10} P_7 = \alpha_3 P_8 \qquad P_{11} P_8 = (\alpha_3 + \alpha_4) P_6 P_9 P_{10} \qquad P_{12} P_{10} = \alpha_4 P_{11}
   \end{gathered}
   $$
  These equalities are  straightforward to check by hand using the explicit values of the $P_j , 1 \leq j \leq 12$.

  \bigskip
 
Unlike the case where $\mathfrak{g}$ is of type $A_n$, the Property (C) is here non trivial a priori. For instance, the positive root $\alpha_1$ appears with multiplicity $2$ in $P_5$. The required inequality is thus guaranteed by the fact that $\alpha_1$ also divides $P_{5_{+}}=P_8$ (with multiplicity $1$). Similarly, one has $(\alpha_1+\alpha_2+\alpha_3 ; P_8) = 2$ and  the inequality (C) follows from the fact that $(\alpha_1+\alpha_2+\alpha_3 ; P_{8_{+}}) = (\alpha_1+\alpha_2+\alpha_3 ; P_{11}) =  1$. 
 
 \bigskip

 Finally, we have proven that the standard seed $\s^{\mathbf{i}_{nat}}$ satisfies the  Properties (A), (B) and (C).

  \section{Cluster theory of homogeneous modules}
  \label{conclu}
   
In this conclusive section, we discuss various evidences of the connections between the determinantal modules categorifying the flag minors of $\mathbb{C}[N]$ and the (prime) strongly homogeneous modules of $R-gmod$ in the sense of Kleshchev-Ram \cite{KRhom}. This leads us to propose a conjectural criterion of primeness of the homogeneous module $S(w)$ ($w \in W$).

   Recall from Section~\ref{fullcom} the subsets $\mathcal{FC}, \mathcal{M}in, \mathcal{M}in^{+}$ of $W$. We will consider certain subsets of these sets, by intersecting them with the set of \textit{strict} elements of $W$ defined as follows. 
   
    \begin{deftn} \label{defstrict}
    An element $w \in W$ is said to be \textit{strict} if for every reduced expression $\mathbf{j} = (j_1, \ldots , j_l) \in \text{Red}(w)$, one has 
    $$ \forall 1 \leq r < l, \exists 1 \leq p \leq r < q \leq l, j_p \cdot j_q \neq 0 . $$
   \end{deftn}
   
In other words, there is no gap in any reduced expression of $w$. We let $W_{0}$ denote the set of all strict elements of $W$ and we set 
   $$ \mathcal{FC}_{0} := W_{0} \cap \mathcal{FC} \quad , \quad  \mathcal{M}in_{0} := W_{0} \cap \mathcal{M}in \quad , \quad  \mathcal{M}in_{0}^{+} := W_{0} \cap  \mathcal{M}in^{+}. $$
   As we saw in Example~\ref{exampleA3}, the determinantal modules whose classes in $\mathbb{C}[N]$ are flag minors coincide with the prime strongly homogeneous when $\mathfrak{g} = \mathfrak{sl}_4$. We also noticed that the list of Weyl group elements parametrizing these modules is exactly $\mathcal{M}in^{+} \setminus \{s_3 s_1 \}$. Thus it is immediate that this list is in fact the list of elements of $\mathcal{M}in_{0}^{+}$. The same observation can be checked for $\mathfrak{g}= \mathfrak{sl}_5$. 
   We propose the following:
  
    \begin{conj} \label{conjA}
   Let $\mathfrak{g}$ be of type $A_n$. The set of cluster variables of the seeds $\mathcal{S}^{{\bf i}_{<}}$ ($<$ running over all the possible orderings on $I$) is exactly the set of isomorphism classes of the strongly homogeneous modules $S(w) , w \in \mathcal{M}in_{0}^{+}$. 
 \end{conj}
  
   As we saw in Section~\ref{initD4}, the determinantal modules corresponding to flag minors are not necessarily homogeneous when $\mathfrak{g}$ is of type $D_4$. However, it is not hard in this case to list the strict dominant minuscule elements of $W$. We obtain the following:
  $$
    \begin{gathered}
    s_1 \quad s_2 \quad s_3 \quad s_4 \\
    s_1s_3 \quad s_3s_1 \quad s_2s_3 \quad s_3s_2 \quad s_4s_3 \quad s_3s_4 \\    
    s_1s_2s_3 \quad s_1s_3s_2 \quad s_4s_2s_3 \quad s_4s_3s_2 \quad s_2s_3s_1 \quad s_2s_3s_4 \quad s_1s_3s_4 \quad s_4s_3s_1 \quad s_1s_4s_3 \\
     s_3s_1s_2s_3 \quad  s_3s_1s_4s_3 \quad  s_3s_2s_2s_3 \quad s_4s_2s_1s_3 \quad s_4s_2s_3s_1 \quad s_1s_2s_3s_4 \quad  s_1s_4s_3s_2 \\
     s_3s_1s_2s_3s_4 \quad s_3s_1s_4s_3s_2 \quad s_3s_2s_4s_3s_1  \\
      s_4s_3s_1s_2s_3s_4  \quad s_2s_3s_1s_4s_3s_2 \quad  s_1s_3s_2s_4s_3s_1 .
    \end{gathered}
    $$
  
  Using Theorem~\ref{thmcasbi2}, one can compute the determinantal modules corresponding to the cluster variables of the seed $\s^{\mathbf{i}}$,  $\mathbf{i}$ coming from a total ordering on $I$. Unlike the type $A_n$ case though, this is not sufficient to get all the homogeneous modules $S(w)$, $w \in \mathcal{M}in_{0}^{+}$. Nonetheless, it is not hard to compute the determinantal  modules of the remaining standard seeds by performing certain well-chosen mutations as in Section~\ref{propag}. Then one can observe that for $w$ in the list above, the isomorphism class of the homogeneous module $S(w)$ is always a flag minor. We propose the following conjecture:
  
  \begin{conj} \label{conjhomog}
  Let $\mathfrak{g}$ be a simple Lie algebra of (finite) simply-laced type. Then for any $w \in \mathcal{M}in_{0}^{+}$, the class of $S(w)$ is a flag minor in $\mathbb{C}[N]$. 
  \end{conj}
  
   \begin{rk}
   Note that the assumption that $w$ is dominant minuscule is crucial. Indeed, consider for instance $\mathfrak{g}$ of type $D_4$ and $w := s_3s_1s_2s_4s_3$ ($3$ is the trivalent node). Then $w$ is  fully-commutative but not dominant minuscule. On the other hand, one can show that the homogeneous module $S(w)$ is not real. In particular by the results of \cite{KKKO} it is not a cluster monomial and a fortiori not a determinantal module.
   \end{rk}
  
  This conjecture would imply the following primeness criterion for strongly homogeneous modules:
  
   \begin{conj} \label{lastconj}
   The prime strongly homogeneous modules in $R-gmod$ are exactly the modules of the form $S(w) , w \in \mathcal{M}in_{0}^{+}$. 
    \end{conj} 
    
 \begin{lem}
 Conjecture~\ref{conjhomog} implies Conjecture~\ref{lastconj}.
  \end{lem}    
    
     \begin{proof}
 Conjecture~\ref{conjhomog} would imply that if $w \in \mathcal{M}in_{0}^{+}$, then $S(w)$ is prime (as it categorifies a cluster variable). But conversely, it is easy to check that if $w$ is not strict, then the module $S(w)$ can be decomposed as a convolution product of two simple modules, and thus it is not prime. Thus it would imply Conjecture~\ref{lastconj}. 
   \end{proof}

 \Address

 \end{document}